\documentclass[11pt]{amsart}
\usepackage{amscd,amssymb,euscript,hyperref}

\newcommand\A{\mathbb A}
\renewcommand\P{\mathbb P}
\renewcommand\phi{\varphi}

\newcommand\gm{\mathop{\mathbb{G}_\mathrm{m}}}

\DeclareMathOperator\Sm{Sm}
\DeclareMathOperator\act{act}
\DeclareMathOperator\Hom{Hom}
\DeclareMathOperator\Gr{Gr}
\DeclareMathOperator\Ob{Ob}
\DeclareMathOperator\Mor{Mor}
\DeclareMathOperator\Spin{Spin}
\DeclareMathOperator\GL{GL}
\DeclareMathOperator\PGL{PGL}
\DeclareMathOperator\SL{SL}
\DeclareMathOperator\can{{can}}
\DeclareMathOperator\Pic{{Pic}}
\DeclareMathOperator\Spec{{Spec}}
\DeclareMathOperator{\Ker}{Ker}
\DeclareMathOperator{\ad}{ad}
\DeclareMathOperator{\Id}{Id}
\DeclareMathOperator{\Aut}{Aut}

\newcommand{\fm}{\mathfrak m}
\newcommand{\fp}{\mathfrak p}

\newcounter{noindnum}[subsection]
\renewcommand{\thenoindnum}{\roman{noindnum}}
\newcommand{\noindstep}{\refstepcounter{noindnum}{\rm(}\thenoindnum\/{\rm)} }
\newcommand{\stepzero}{\setcounter{noindnum}{0}
}
\theoremstyle{definition}
\newtheorem{definition}{Definition}[section]

\theoremstyle{plain}
\newtheorem{theorem}{Theorem}

\newtheorem{proposition}[definition]{Proposition}
\newtheorem{lemma}[definition]{Lemma}
\newtheorem{corollary}[definition]{Corollary}

\theoremstyle{remark}
\newtheorem{remark}[definition]{Remark}
\newtheorem{remarks}[definition]{Remarks}

\newcommand\cE{{\mathcal{E}}}
\newcommand\cF{{\mathcal{F}}}

\newcommand\cO{{\mathcal{O}}}

\newcommand\bF{{\mathbf F}}
\newcommand\bG{{\mathbf G}}
\newcommand\bH{{\mathbf H}}

\newcommand\bZ{{\mathbf Z}}

\newcommand\bx{{\mathbf x}}
\newcommand\by{{\mathbf y}}

\newcommand\Z{\mathbb Z}

\title[On the purity conjecture of Nisnevich]{On the purity conjecture of Nisnevich for torsors under reductive group schemes}

\author{Roman Fedorov}
\email{fedorov@pitt.edu}
\address{University of Pittsburgh, PA, USA}

\begin{document}
\begin{abstract}
    Let $R$ be a regular semilocal integral domain containing an infinite field $k$. Let $f\in R$ be an element such that for all maximal ideals $\fm$ of $R$ we have $f\notin\fm^2$. Let $\bG$ be a reductive group scheme over~$R$. Under an isotropy assumption on $\bG$ we show that a $\bG$-torsor over the localization $R_f$ is trivial, provided that it is rationally trivial. We show that this is not true without the isotropy assumption. Finally, if $\bG$ is a commutative group scheme of multiplicative type and the regular semilocal ring contains a field of characteristic zero, we prove an analogue of Nisnevich purity conjecture for higher \'etale cohomology groups.

    The first statement is derived from its abstract version concerning presheaves of pointed sets satisfying some properties. The counterexample is constructed by providing a torsor over a local family of affine lines that cannot be extended to a torsor over the corresponding local family of projective lines. The latter is accomplished using the technique of affine Grassmannians.

    \vspace{3pt}

    \noindent {\sc R\'esum\'e.} Soit $R$ anneau int\`egre semi-local r\'egulier contenant un corps $k$ infini. Soit $f\in R$ un \'el\'ement tel que pour tous les id\'eaux maximaux $\fm$ de $R$ on a $f\notin\fm^2$. Soit $\bG$ un sch\'ema en groupes r\'eductifs sur $R$. Sous une hypoth\`ese d'isotropie sur $\bG$, nous montrons qu'un $\bG$-torseur sur le localis\'e $R_f$ est trivial, \`a condition qu'il soit rationnellement trivial. Nous montrons que cette affirmation ne tient pas sans l'hypoth\`ese d'isotropie. Enfin, si $\bG$ est un sch\'ema en groupes commutatifs de type multiplicatif et l'anneau semi-local r\'egulier contient un corps de caract\'eristique z\'ero, nous prouvons un analogue de la conjecture de puret\'e de Nisnevich pour les cohomologies \'etales sup\'erieures.

    La premi\`ere affirmation est d\'eriv\'ee de sa version abstraite concernant les pr\'efaisceaux d'ensembles point\'es satisfaisant certaines propri\'et\'es. Le contre-exemple est obtenu en construisant un torseur sur une famille locale de droites affines qui ne s'\'etend pas en un torseur au-dessus de la famille correspondante de droites projectives. Cette derni\`ere est accomplie \`a l'aide de la technique des grassmanniennes affines.

\end{abstract}

\maketitle

\section{Introduction and main results}
Let $R$ be a regular local ring, let $\bG$ be a reductive group scheme over $R$, and let~$\cE$ be a $\bG$-torsor over $R$. The famous Grothendieck--Serre conjecture predicts that $\cE$ is trivial, if it is trivial over the fraction field $K$ of $R$. This has been proved in the case, when $R$ contains a field~\cite{FedorovPanin,PaninFiniteFieldsIzvestiya}. In fact, a generalization of this conjecture to \emph{semilocal\/} regular integral domains containing fields has been proved. The conjecture is still wide open in the mixed characteristic case, though recently there has been some significant progress~\cite{FedorovMixedChar,CesnaviciusGrSerre,GuoPaninGrSerre,GuoFeiGrSerre}.

The conjecture~\cite[Conj.~1.3]{Nisnevich2}, that was formulated by Nisnevich in 1989, is a generalization of the Grothendieck--Serre conjecture: \emph{let~$R$ and $\bG$ be as before, let $f\in R$ be such that $f\notin\fm^2$, where $\fm$ is the maximal ideal of $R$, and let~$\cE$ be a $\bG$-torsor over $R_f$, then $\cE$ is trivial, if it is trivial over~$K$.} This conjecture reduces to the Grothendieck--Serre conjecture if $f=1$. Very little was known about this conjecture until now. If $R$ is a discrete valuation ring, the conjecture reduces to the Grothendieck--Serre conjecture if $f$ is invertible and is trivial otherwise. Nisnevich in~\cite{Nisnevich2} proved the conjecture in dimension two, assuming that the residue field of the local ring $R$ is infinite and that $\bG$ is quasi-split (the condition that the residue field is infinite is removed by \v{C}esnavi\v{c}ius, see~\cite[Sect.~3.4.2(3)]{CesnaviciusProblems}). Note that the Nisnevich conjecture is non-trivial even for $\bG=\GL_n$, as opposed to the Grothendieck--Serre conjecture (in the $\GL_n$ case the conjecture goes back to Quillen, see~\cite{QuillenOnSerre}). In this case the conjecture was proved by
Bhatwadekar and Rao in~\cite{BhatwadekarRaoOnQuillen}, assuming that $R$ is a localization of a finitely generated algebra over a field. This was extended to arbitrary regular local rings containing fields by Popescu~\cite[Thm.~1]{PopescueOnNisnevich}. Guo proved a version of this conjecture for the rings of formal power series over valuation rings~\cite[Cor.~7.5]{Guo2021GrSerreValRings}.

In this article, if the regular local ring contains an infinite field, we prove the conjecture under some isotropy condition on $\bG$. In fact, under these assumptions, we prove the generalization of the conjecture to semilocal regular integral domains. Then we show that the conjecture does not hold without the isotropy assumption.

Finally, if $\bG$ is a commutative group scheme of multiplicative type and the regular semilocal ring contains a field of characteristic zero, we prove an analogue of Nisnevich purity conjecture for higher \'etale cohomology groups.

\subsection{Notations and conventions}
If $C$ is a category, we denote by $\Ob(C)$ its class of objects.

If $A$ is a pointed set, we denote its distinguished element by $\star$.

If $X$ is an integral scheme, we denote by $K(X)$ its function field. If $x$ is a schematic point of~$X$, we denote by $\cO_{X,x}$ its local ring and by $\fm_x\subset\cO_{X,x}$ the corresponding maximal ideal. More generally, if $\bx\subset X$ is a set of schematic points contained in an open affine subset of $X$, we denote by $\cO_{X,\bx}$ the corresponding semilocal ring.

We often work over a fixed base field $k$, in that case the product is the product over $k$: $\times:=\times_k$.

Let $U$ be a scheme and $\bG$ be a flat $U$-group scheme finitely presented over $U$. A $U$-scheme $\cE$ with a left action $\act\colon\bG\times\cE\to\cE$ is called \emph{a $\bG$-torsor over $U$\/} if~$\cE$ is faithfully flat and quasi-compact over $U$ and the action is simply transitive, that is, the morphism $(\act,p_2)\colon\bG\times_U\cE\to\cE\times_U\cE$ is an isomorphism (see~\cite[Sect.~6]{GrothendieckFGA}). A $\bG$-torsor $\cE$ over $U$ is \emph{trivial\/} if $\cE$ is isomorphic to $\bG$ as a~$U$-scheme with an action of $\bG$. This is well-known to be equivalent to the projection $\cE\to U$ having a section. If $T$ is a $U$-scheme, we will use the term ``$\bG$-torsor over $T$'' to mean a ``$\bG_T$-torsor over $T$''.

Assume that $\bG$ is affine over $U$. The pointed set of isomorphism classes of $\bG$-torsors over $T$ is the first non-abelian cohomology group $H^1(T,\bG):=H^1_{\mathrm{fppf}}(T,\bG)$ (note that, since $\bG$ is affine, every class in $H^1_{\mathrm{fppf}}(T,\bG)$ gives rise to a $\bG$-torsor because affine schemes can be glued in the fppf topology). The $\bG$-torsors are also called principal $\bG$-bundles.

\subsection{Nisnevich purity conjecture}\label{sect:NisConj}
Let us recall a notion from~\cite[Def.~1.1]{FedorovGrSerreNonSC}. Let $\bG$ be a semisimple group scheme of adjoint type over a connected scheme $U$ (see~\cite[Exp.~XXII, Def.~4.3.3]{SGA3-3}). By~\cite[Exp.~XXIV, Prop.~5.10]{SGA3-3} there is a sequence $U_1,\ldots,U_r$ of finite \'etale connected $U$-schemes such that
\begin{equation}\label{eq:prod}
    \bG\simeq\prod_{i=1}^r\bG^i,
\end{equation}
where $\bG^i$ is the Weil restriction of a simple adjoint $U_i$-group scheme (see~\cite[Exp.~XXIV, Sect.~5.3]{SGA3-3} for the definition of a simple group scheme). Note that the group schemes $\bG^i$ are uniquely defined by $\bG$ up to isomorphism. Recall that a semisimple group scheme $\bG$ over a scheme $S$ is called \emph{isotropic\/}, if it contains a non-trivial split torus. If $S$ is connected and semilocal, then a semisimple group scheme $\bG$ is isotropic if and only if it contains a proper parabolic subgroup scheme (see~\cite[Exp.~XXVI, Cor.~6.14]{SGA3-3}).

\begin{definition}
We say that a semisimple $U$-group scheme $\bG$ of adjoint type is \emph{strongly locally isotropic\/} if each factor~$\bG^i$ of $\bG$ is isotropic Zariski locally over $U$.
\end{definition}

We send the reader to~\cite[Rem.~1.2]{FedorovGrSerreNonSC} for the discussion of this notion.

Let now $\bG$ be a reductive $U$-group scheme. Let $\bZ$ be the center of $\bG$. We have the adjoint group scheme $\bG^{\ad}:=\bG/\bZ$ (see~\cite[Exp.~XXII, Prop.~4.3.4(i), Prop.~4.3.5(ii), Def.~4.3.6]{SGA3-3}).

We prove the Nisnevich conjecture assuming that the regular semilocal integral domain contains an infinite field and the adjoint group scheme of $\bG$ is strongly locally isotropic.

\begin{theorem}\label{th:Nisnevich}
Assume that $R$ is a regular semilocal integral domain containing an infinite field $k$. Let $K=K(R)$ be the quotient field of $R$. Let $f\in R$ be such that $f\notin\fm^2$ for any maximal ideal $\fm$ of $R$. Let $\bG$ be a reductive $R$-group scheme such that $\bG^{\ad}$ is strongly locally isotropic. Let $\cE$ be a $\bG$-torsor over $R_f$ such that its pullback to~$K$ is trivial. Then $\cE$ is trivial. In other words, the natural map $H^1(R_f,\bG)\to H^1(K,\bG)$ has a trivial kernel.
\end{theorem}
\begin{remark}
    We note that our proof was extended by \v{C}esnavi\v{c}ius in~\cite{CesnaviciusBassQuillen} to the case when $k$ is allowed to be finite. It is also shown in loc.~cit.~that it is enough to assume that $\bG$ is strongly locally isotropic along the hypersurface $\{f=0\}$.
\end{remark}
We will axiomatize the properties of the first non-abelian cohomology functor needed to prove the statement (see Theorem~\ref{th:NisnevichFun} below). We derive the above theorem from its ``abstract'' version in Section~\ref{sect:ProofNis}.

The Nisnevich conjecture is not true in general. We indeed provide counterexamples when $\bG^{\ad}$ fails to be strongly locally isotropic.

\begin{theorem}\label{th:CounterExample}
  There are a regular local ring $R$ containing an infinite field, a simple simply-connected group scheme $\bG$ over $R$, and a generically trivial $\bG$-torsor $\cE$ over $\A^1_R$ that cannot be extended to $\P^1_R$.
\end{theorem}

This theorem, proved in Section~\ref{sect:CounterEx} below, contradicts the Nisnevich conjecture. Indeed, let $u$ be the closed point of $\Spec R$. Were the Nisnevich conjecture true, we could apply it to the local ring of $\infty\times u$ on $\P^1_R$, to find a Zariski neighborhood $W$ of the infinity divisor $\infty\times\Spec R\subset\P_R^1$ such that $\cE|_{W-(\infty\times\Spec R)}$ is trivial. Then, gluing $\cE$ with the trivial torsor over $W$, we would get an extension of $\cE$ to $\P^1_R$.

\begin{remark}
  The regular local ring in Theorem~\ref{th:CounterExample} can be constructed as follows (see Section~\ref{sect:CounterEx} and~\cite{FedorovExotic}). Let $k$ be an algebraically closed field and consider a closed embedding of $k$-groups $\Spin_7\to\GL_n$ for some $n$. Then $R$ is the local ring of a closed point of $\A^1_k\times(\GL_n/\Spin_7)$. In this case, the group scheme $\bG$ is a strongly inner form of $\Spin_7$.
\end{remark}

\subsubsection{Higher cohomology} In the case when the group scheme $\bG$ is commutative, we have a version of Theorem~\ref{th:Nisnevich} above for higher \'etale cohomology.

\begin{theorem}\label{th:NisnevichAbelian}
Assume that $R$ is a regular semilocal integral domain containing a field of characteristic zero. Let $K=K(R)$ be the quotient field of $R$. Let $f\in R$ be such that $f\notin\fm^2$ for any maximal ideal $\fm$ of $R$. Let $\bG$ be an $R$-group scheme of multiplicative type. Then for all $n\ge1$ the natural homomorphism $H^n(R_f,\bG)\to H^n(K,\bG)$ is injective.
\end{theorem}

If $f=1$ and $\bG$ comes from the base field, this is~\cite[Thm.~4.2]{ColliotTeleneOjanguren}. This theorem will be derived from Theorem~\ref{th:NisnevichFun} as well; see Section~\ref{sect:ProofAbelian}.

\subsection{Grothendieck--Serre and Nisnevich conjectures for Nisnevich semisheaves}
Colliot--Th\'el\`ene and Ojanguren in~\cite[Thm~1.1]{ColliotTeleneOjanguren} proved the conjecture of Grothendieck and Serre for group schemes coming from the base field, assuming additionally that the base field is infinite. Roughly speaking, they proved the conjecture for any presheaf of pointed sets that satisfies the Nisnevich descent, commutes with filtered limits, and is homotopy invariant. Then they have shown that the presheaf of first non-abelian cohomology satisfies these properties.

In~2012 Jean-Pierre Serre asked whether it is possible to recast the proof in~\cite{FedorovPanin} in the same way. We answer this question positively in Section~\ref{sect:GrSerre}. Let us introduce the relevant terminology.

\subsubsection{Essentially smooth morphisms}
\begin{definition}
    Let $X$ be a Noetherian scheme. We say that a Noetherian $X$-scheme $Y$ is \emph{essentially smooth\/} over $X$ if it is a filtered inverse limit of smooth $X$-schemes with transition morphisms being open affine embeddings.
\end{definition}

For a discussion of this notion, see Section~\ref{sect:var}. We denote by $\Sm'_X$ the full subcategory of the category of $X$-schemes whose objects are Noetherian schemes essentially smooth over $X$. We emphasize that the morphisms in $\Sm'_X$ are arbitrary $X$-morphisms of schemes.

An important example of an essentially smooth scheme is the following: let $X$ be an affine scheme, let $\bx$ be a finite set of schematic points of $X$, let $\cO_{X,\bx}$ be the semilocal ring of $\bx$, then $\Spec\cO_{X,\bx}$ is an object of $\Sm'_X$. In particular, if $X$ is integral, then $\Spec K(X)$, where $K(X)$ is function field of $X$, is essentially smooth over $X$.

\begin{lemma}\label{lm:EssSm}
(i) Let $X$, $Y$, and $Z$ be Noetherian schemes such that $Y$ be essentially smooth over $X$ and $Z$ be essentially smooth over $Y$. Then $Z$ is essentially smooth over $X$.

(ii) A base change of an essentially smooth morphism is essentially smooth.
\end{lemma}
\begin{proof}
(i) Consider first the case, when $Z$ is smooth over $Y$. By definition of essentially smooth morphisms, we can write $Y=\lim\limits_{\longleftarrow} Y_\alpha$, where each $Y_\alpha$ is smooth over $X$ and the transition morphisms are open affine embeddings. We can find $\alpha$ and a $Y_\alpha$-scheme $Z_\alpha$ such that $Z=Z_\alpha\times_{Y_\alpha}Y$ (see~\cite[Prop.~17.7.8(ii)]{EGAIV.4}). For $\beta\succeq\alpha$, set $Z_\beta:=Z_\alpha\times_{Y_\alpha}Y_\beta$. Then $Z=\lim\limits_{\longleftarrow} Z_\beta$. Note that there is $\gamma$ such that for $\beta\succeq\gamma$ the scheme $Z_\beta$ is smooth over $Y_\beta$ (indeed, note that the projection $Z_\alpha\to Y_\alpha$ is smooth over the points of $Y$, then use openness of smoothness and the assumption that our schemes are Noetherian). Then $Z_\beta$ is also smooth over $X$, and we see that $Z$ is essentially smooth over $X$.

Now consider the general case. Write $Z=\lim\limits_{\longleftarrow} Z_\alpha$, where $Z_\alpha$ are smooth over $Y$ and the transition morphisms are open affine embeddings. Fix some $\alpha_0$ and set $Z':=Z_{\alpha_0}$. We may assume that all $Z_\alpha$ are open subschemes of $Z'$. By the already proved, $Z'$ is essentially smooth over $X$, so we can write $Z'=\lim\limits_{\longleftarrow} Z'_\beta$, where $Z'_\beta$ are smooth over $X$ and the morphisms are open affine embeddings. Again, we may assume that all $Z'_\beta$ are open subschemes of a fixed $Z''$ smooth over $X$. For each $\alpha$, since $Z_\alpha$ is open in $Z'$, we can find $\beta(\alpha)$ and an open subscheme $W_\alpha\subset Z'_{\beta(\alpha)}$ such that
\[
    Z_\alpha=\lim\limits_{\beta\succeq\beta(\alpha)}W_\alpha\cap Z'_\beta.
\]
(Indeed, we reduce to the case, when $Z''$ is affine, in which case $Z'_\beta$, $Z'$, and $Z_\alpha$ are affine as well.) Now $Z$ is the limit of the filtered system $W_\alpha\cap Z'_\beta$, and we see that it is essentially smooth over~$X$.

(ii) is clear.
\end{proof}

\subsubsection{Nisnevich semisheaves}
Torsors can be glued in the Nisnevich topology. However, torsors that are isomorphic locally in the Nisnevich topology need not be isomorphic. Thus, the presheaf of first non-abelian cohomology satisfies only one part of the sheaf axiom. We call such presheaves \emph{semisheaves}. Let us give the definitions.

\begin{definition}\label{def:ElemDistSq}
An \emph{elementary distinguished square\/} is a Cartesian diagram of schemes:
\begin{equation}\label{eq:ElemDistSq}
\begin{CD}
W @>>> V\\
@VVV @VV p V\\
U @>j>> Y,
\end{CD}
\end{equation}
where $p$ is \'etale, $j$ is an open embedding, and $p^{-1}(Y-U)\to Y-U$ is an isomorphism, where we equip the corresponding closed subsets with the reduced scheme structures.
\end{definition}

\begin{definition}\label{def:ElemDistSqDescend}
Let $X$ be a Noetherian scheme and $\bF$ be a presheaf of pointed sets on $\Sm'/X$. We say that $\bF$ is a \emph{Nisnevich semisheaf of pointed sets\/} if for any elementary distinguished square~\eqref{eq:ElemDistSq} in $\Sm'/X$ the corresponding map
\[
    \bF(Y)\to\bF(U)\times_{\bF(W)}\bF(V)
\]
is surjective.
\end{definition}

\begin{remarks}\label{rem:Zariski}
(i) If $Y=Y_1\cup Y_2$ is a Zariski cover of $Y$, then
\[\begin{CD}
Y_1\cap Y_2 @>>> Y_1\\
@VVV @VVV\\
Y_2 @>>> Y
\end{CD}\]
is an elementary distinguished square. Thus, every Nisnevich semisheaf also satisfies gluing in the Zariski topology.

(ii) Our definition of a Nisnevich semisheaf is similar to~\cite[Ch.~3, Def.~1.3]{MorelVoevodsky} except that our schemes are essentially smooth rather than smooth over $X$, and we consider semisheaves of pointed sets rather than sheaves of sets. \end{remarks}

\subsubsection{A Grothendieck--Serre type theorem for Nisnevich semisheaves}\label{sect:GrSerre}
Let $X$ be a smooth integral affine $k$-scheme, where $k$ is a field. Let $\bF$ be a Nisnevich semisheaf of pointed sets on $\Sm'/X$. We will formulate the properties under which $\bF$ satisfies a Grothendieck--Serre type statement.

\textbf{(Lim)} $\bF$ commutes with filtered inverse limits, provided that the transition morphisms are open embeddings.

The next property is a version of a local triviality of $\bF$. If $\pi\colon Y\to X$ is an essentially smooth morphism used to view $Y$ as an object of $\Sm'/X$, then we denote the value of $\bF$ on $Y$ by $\bF(Y,\pi)$, when we need to emphasize the dependence on $\pi$.

\textbf{(LT)} Assume that $W$ is an affine integral semilocal $k$-scheme such that all its closed points are finite over $k$, and let $U\subset W$ be a closed subscheme. Assume that we are given two essentially smooth $k$-morphisms $p_1,p_2\colon W\to X$ such that $p_1|_U=p_2|_U$, and this morphism is essentially smooth. Then there is a
finite \'etale $k$-morphism $\pi\colon W'\to W$ with a section $\Delta\colon U\to W'$ and an isomorphism of semisheaves $(p_1\circ\pi)^*\bF\simeq(p_2\circ\pi)^*\bF$ restricting to the identity isomorphism on $U$.

This property requires some explanation. First of all, $p_i\circ\pi$ are essentially smooth by Lemma~\ref{lm:EssSm}, so the pullbacks of semisheaves make sense (again, by Lemma~\ref{lm:EssSm}). The isomorphism amounts to compatible bijections for any essentially smooth $\psi\colon T\to W'$:
\[
    \bF(T,p_1\circ\pi\circ\psi)\xrightarrow{\simeq}\bF(T,p_2\circ\pi\circ\psi).
\]
Finally, the condition that this isomorphism restricts to the identity on $U$ amounts to having for every~$\psi$ as above a commutative diagram
\begin{equation}\label{eq:CD}
\begin{CD}
\bF(T,p_1\circ\pi\circ\psi)@>\simeq>>\bF(T,p_2\circ\pi\circ\psi)\\
@VVV @VVV\\
\bF(T\times_{W'}U,p_1\circ\pi\circ\psi\circ pr_1)@=\bF(T\times_{'W}U,p_2\circ\pi\circ\psi\circ pr_1),
\end{CD}
\end{equation}
where $pr_1$ denotes the projection $T\times_{W'}U\to T$. Here the bottom identification comes from the fact that both morphisms are equal:
\[
    p_i\circ\pi\circ\psi\circ pr_1=p_i\circ\pi\circ\Delta'\circ pr_2=(p_i|_U)\circ pr_2.
\]
(In particular, this morphism is essentially smooth.)

The last property we need is a ``Section property'', which is weaker, than homotopy invariance (at least, if the field $k$ is infinite).

\textbf{(Sec)} Let $U\in\Ob(\Sm'/X)$ be an affine integral semilocal scheme such that all its closed points are finite over $k$. Assume that $Z$ is a closed subscheme of $\A^1_U$ finite over~$U$. Let
\[
    \cE\in\Ker\Bigl(\bF(\A^1_U)\to\bF(\A^1_U-Z)\Bigr),
\]
then for every section $\Delta\colon U\to\A^1_U$ of the projection $\A^1_U\to U$ we have $\Delta^*\cE=\star$.

\begin{theorem}\label{th:GrSerreFun}
Let $X$ be a smooth integral affine $k$-scheme, where $k$ is a field. Let $\bF$ be a Nisnevich semisheaf on $\Sm'/X$ satisfying properties (Lim), (LT), and (Sec) above. Then for any finite set of schematic points $\by\subset X$ the map $\bF(\cO_{X,\by})\to\bF(K(X))$, where $K(X)$ is the fraction field of $X$, has a trivial kernel.
\end{theorem}

We prove this theorem in Section~\ref{sect:EndPrfGrSerreFun} after some preliminaries. The utility of working with semisheaves is demonstrated by the following important corollary. Consider a family version of property~(Sec):

  \textbf{(SecF)} Let $U\in\Ob(\Sm'/X)$ be an integral affine semilocal scheme such that all its closed points are finite over $k$. Assume that $Z$ is a closed subscheme of $\A^1_U$ finite over~$U$. Let $Y\in\Ob(\Sm'/k)$ be an affine scheme, and
   \[
    \cE\in\Ker\Bigl(\bF(\A^1_U\times_kY)\to\bF((\A^1_U-Z)\times_kY)\Bigr),
   \]
   then for every section $\Delta\colon U\times_kY\to\A^1_U\times_kY$ of the projection $\A^1_U\times_kY\to U\times_kY$ we have $\Delta^*\cE=\star$.

\begin{corollary}
\label{cor:Family}
  In the notation of the theorem assume that the Nisnevich semisheaf $\bF$ satisfies properties (Lim), (LT), and (SecF) above. Then for any finite set of schematic points $\by\subset X$ and any affine scheme $Y\in\Ob(\Sm'/k)$ the map \[\bF(\cO_{X,\by}\otimes_kk[Y])\to\bF(K(X)\otimes_kk[Y]),\] where $K(X)$ is the fraction field of $X$, has a trivial kernel.
\end{corollary}
\begin{proof}
  Consider the functor $\bF'$ sending a scheme $T\in\Ob(\Sm'/X)$ to $\bF(T\times_k Y)$ (we note that $T\times_kY$ is essentially smooth over $X$ by Lemma~\ref{lm:EssSm}). It is clear that it is a Nisnevich semisheaf satisfying properties (Lim) and (Sec).

  Let us prove (LT). Let $U\subset W$ and $p_1,p_2\colon W\to X$ be as in the formulation of (LT). Since $\bF$ satisfies (LT), there is a finite \'etale $k$-morphism $\pi\colon W'\to W$ with a section $\Delta\colon U\to W'$ and an isomorphism of semisheaves $(p_1\circ\pi)^*\bF\simeq(p_2\circ\pi)^*\bF$ restricting to the identity isomorphism on $U$.

  Let $\psi\colon T\to W'$ be an essentially smooth morphism. Then the composition $\tau\colon T\times_kY\to W'\times_kY\to W'$ is also essentially smooth, so we get an isomorphism
  \[
    \bF'(T,p_1\circ\pi\circ\psi)=\bF(T\times_kY,p_1\circ\pi\circ\tau)
    \xrightarrow{\simeq}\bF(T\times_kY,p_2\circ\pi\circ\tau)=\bF'(T,p_2\circ\pi\circ\psi).
  \]
  These isomorphisms gives rise to an isomorphism $(p_1\circ\pi)^*\bF'\simeq(p_2\circ\pi)^*\bF'$ restricting to the identity isomorphism on $U$.
\end{proof}

This corollary should be compared to~\cite[Thm.~1.1]{PaninStavrovaVavilov}, \cite[Thm.~7.1]{PaninNiceTriples}, and~\cite[Thm.~1]{FedorovExotic}.

\subsubsection{Nisnevich conjecture for Nisnevich semisheaves}\label{sect:NisConjSemisheaves} The following is one of the main result of this paper.

\begin{theorem}\label{th:NisnevichFun}
Let $X$ be a smooth integral affine $k$-scheme, where $k$ is an infinite field. Let $\bF$ be a Nisnevich semisheaf on $\Sm'/X$ satisfying properties (Lim), (LT), and (SecF) above as well as:

  {\rm\textbf{(A1F)}} Let $K$ be the function field of $X$. Then the map $\bF(K[t])\to\bF(K(t))$ has a trivial kernel.

Let $\by\subset X$ be a finite set of schematic points, let $f\in k[X]$ be a non-zero function such that its zero set $\{f=0\}$ is smooth over $k$. Then the map
  \[
    \bF((\cO_{X,\by})_f)\to\bF(K(X)),
  \]
  where $K(X)$ is the fraction field of $X$, has a trivial kernel.
\end{theorem}

The notation ``(A1F)'' stands for ``affine lines over fields''. It seems plausible that the condition that $k$ be infinite can be removed by applying the technique of~\cite{CesnaviciusBassQuillen}.

\subsection{Variants}\label{sect:var}
  Our definition of ``essentially smooth'' and the decision to work on $\Sm'/X$ are somewhat ad hoc. We discuss possible variations here.

  Some authors would define ``essentially smooth'' as the limit of smooth schemes, where transition morphisms are \'etale. This definition would work for us as well. In fact, we can work with any category larger than $\Sm'/X$. For example, we can start with a Nisnevich semisheaf $\bF$ on the category of \emph{all\/} schemes over~$X$. The reason is that torsors are defined over any scheme. On the other hand, that would make our Theorems~\ref{th:GrSerreFun} and~\ref{th:NisnevichFun} formally weaker.

  We can work with just affine essentially smooth $X$-schemes. That would require additional checks. For example, we would have to make sure that the schemes $Y$ and $Z$ in Section~\ref{sect:NiceTriples} have affine complements. This is possible but is not very convenient. If we follow this ``affine'' approach, we can also assume that~$U$ in Definition~\ref{def:ElemDistSqDescend} is a principal open subset of $X$, that is, of the form $X_f$ for $f\in k[X]$. The required gluing property for torsors follows then from~\cite[Prop. 2.6(iv)]{ColliotTeleneOjanguren}.

  In principle, we can work with functors defined on $X$-schemes of finite type. In this case, we must either extend $\bF$ to essentially smooth schemes using (Lim) for definition, or repeatedly shrink our schemes. Note that then even the formulation of Theorem~\ref{th:GrSerreFun} should be changed to the following: if $\cE\in\bF(X)$ is such that for some non-empty open subscheme $X'\subset X$ we have $\cE|_{X'}=\star$, then there is an open subscheme $X''\subset X$ \emph{containing} $\by$ such that $\cE|_{X''}=\star$. Theorem~\ref{th:NisnevichFun} should be changed similarly.

  We can replace $\cO_{X,\by}$ in Theorems~\ref{th:GrSerreFun} and~\ref{th:NisnevichFun} by arbitrary integral semilocal domains geometrically regular over $X$ at the price of strengthening (Lim) and using Popescu's Theorem.

  \subsection{Acknowledgements} The author is thankful to Vladimir Chernousov, Philippe Gille, Daniel Krashen, and Ivan Panin for stimulating discussions. The author is thankful to Ofer Gabber, Anastasia Stavrova, and to the referee for useful comments and for finding imprecisions in early versions. The author is partially supported by NSF DMS grants 1406532, 1764391, and 2001516. A~part of the work was done during the author's stay at Max Planck Institute for Mathematics in Bonn.

\section{Proofs of Theorem~\ref{th:GrSerreFun} and Theorem~\ref{th:NisnevichFun}}
In Sections~\ref{sect:Fib}--\ref{sect:DescendA1} we reduce the theorems to questions about affine lines over semilocal rings. This is quite standard and follows the usual strategy of the proof of the Grothendieck--Serre conjecture. Some care is needed because we work with abstract functors and we also have the hypersurface $\{f=0\}$. The proof of Theorem~\ref{th:GrSerreFun} is an easy consequence of this. To prove  Theorem~\ref{th:NisnevichFun}, we need an additional result (Proposition~\ref{pr:A1}), which seems to be totally new. This is where we need additional assumptions that $k$ is infinite. It is curious that in the proof of Proposition~\ref{pr:A1} we are using Corollary~\ref{cor:Family} of Theorem~\ref{th:GrSerreFun}.

\subsection{Fibrations into curves}\label{sect:Fib}
The main goal of this section is to prove the following proposition.
\begin{proposition}\label{pr:Fibration}
  Let $X$ be an integral affine scheme smooth over a field $k$, let $\bx\subset X$ be a non-empty finite set of closed points, let $f\in k[X]$ be such that $f\ne0$ and the hypersurface $\{f=0\}$ is smooth over $k$, and let $Z\subsetneq X$ be a closed subscheme. Then there are:

\stepzero\noindstep An open affine subscheme $S\subset\A^{\dim X-1}_k$. \\
\noindstep An open affine subscheme $X'\subset X$ containing $\bx$ and a smooth morphism $X'\to S$ of pure relative dimension one such that $\{f|_{X'}=0\}$ is $S$-\'etale and $S$-finite, $Z\cap X'$ is $S$-finite.
\end{proposition}

If the field $k$ is algebraically closed, the statement is similar to Artin~\cite[Exp.~XI, Prop.~3.3]{SGA4-3} and can be proved along the same lines. Some new ideas are needed if $k$ is finite, or if the residue fields of the points of $\bx$ are not separable over~$k$. We follow~\cite{FedorovMixedChar}. For integers $l_0$, $l_1$, \ldots, $l_{n-1}$ let $\P_k(l_0,l_1,\ldots,l_{n-1})$ denote the corresponding $(n-1)$-dimensional weighted projective space, see~\cite[Sect.~3.6]{FedorovMixedChar}. We start the proof with an auxiliary proposition.

\begin{proposition}\label{pr:GoodNeighb}
Let $k$ be a field, let $N_1$ be a positive integer, and let $\overline X\subset\P_k^{N_1}$ be a closed subscheme of pure dimension $n$, let $X\subset\overline X$ be an open subscheme smooth over $k$ and let $\bx\subset X$ be a non-empty finite set of closed points. Let $T_1$ and~$T_2$ be closed subsets of $\overline X$ of dimensions at most $n-1$ and $n-2$ respectively such that $\bx\cap T_2=\emptyset$. Let~$\overline Y$ be a closed codimension one subset of $\overline X$ such that $\overline Y\cap X$ is smooth over $k$. For an integer~$r$ consider the $r$-fold Veronese embedding $\P_k^{N_1}\hookrightarrow\P_k^{N_r}$ and identify $\overline X$ with a closed subscheme of $\P_k^{N_r}$, using this embedding. Then there are a positive integer $r$ and sections $\sigma_0\in H^0(\P_k^{N_r},\cO(1))$, $\sigma_1\in H^0(\P_k^{N_r},\cO(l_1))$, \ldots, $\sigma_{n-1}\in H^0(\P_k^{N_r},\cO(l_{n-1}))$ for some positive integers $l_i$ such that for all $x\in\bx$ we have

\stepzero\noindstep\label{i} $\sigma_0(x)\ne0$.

\noindstep\label{GoodNeighb_Transversal} Let $\psi\colon\P^{N_r}_k\dashrightarrow\P_k(1,l_1,\dotsc,l_{n-1})$ be the rational morphism defined by the sections $\sigma_i$. Then the subscheme $\psi^{-1}(\psi(x))\cap X$ is smooth of dimension one over $k(\psi(x))$, while $\psi^{-1}(\psi(x))\cap X\cap\overline Y$ is \'etale over $k(\psi(x))$.

\noindstep\label{iii} $\psi^{-1}(\psi(x))\cap T_1$ is finite.

\noindstep\label{iv} $\psi^{-1}(\psi(x))\cap T_2=\emptyset$.

\noindstep\label{v} $\{\sigma_0=\sigma_1=\dotsb=\sigma_{n-1}=0\}\cap  T_i=\emptyset \text{ for }i=1,2$.
\end{proposition}
If $\bx$ consists of a single point and $\overline Y=\emptyset$, then this follows from~\cite[Prop.~3.13]{FedorovMixedChar}. The proof in our case is just a minor modification. We present the proof for the sake of completeness. The proof in the finite field case is significantly different from the proof in the infinite field case.

\begin{proof}[Proof of Proposition~\ref{pr:GoodNeighb} in the case, when $k$ is finite.]
We recall a statement from~\cite{FedorovMixedChar}.
\begin{lemma}\label{lm:Poonen}
     Assume that $T_1$, \ldots, $T_n$ are locally closed subschemes of $\P_k^N$, where $k$ is a finite field. Let $T'\subset T$ be smooth locally closed subschemes of $\P_k^N$, and let~$F$ be a~finite set of closed points of $\P_k^N$. Assume that for all $i$ such that $T_i$ is finite, we have $T_i\cap F=\emptyset$. Then there is a~hypersurface $H\subset\P_k^N$ such that the scheme theoretic intersections $H\cap T$ and $H\cap T'$ are smooth, $F\subset H$, and for $i=1,\dotsc,n$ we have $\dim(H\cap T_i)<\dim T_i$ or $H\cap T_i=\emptyset$.
\end{lemma}
\begin{proof}
By dropping empty sets, we may assume that all $T_i$ are nonempty. Then the statement becomes~\cite[Prop.~3.12]{FedorovMixedChar}.
\end{proof}

We return to the proof of Proposition~\ref{pr:GoodNeighb} in the finite field case. Let us define the dimension of the empty scheme to be $-1$. We inductively construct $\sigma_0$, \ldots, $\sigma_{n-1}$ such that for all $x\in\bx$ we have $\sigma_0(x)\ne 0$, for $m=1,\dotsc,n-1$ we have $\sigma_m(x)=0$,
\begin{gather}
    \dim(\{\sigma_0=\dotsb=\sigma_m=0\}\cap T_i)<n-m-i,\notag\\
    \dim(\{\sigma_1=\dotsb=\sigma_m=0\}\cap T_i)\le n-m-i,\label{eq:DimEst}
\end{gather}
and the intersections
\[
    \{\sigma_1=\dotsb=\sigma_m=0\}\cap X\text{ and }\{\sigma_1=\dotsb=\sigma_m=0\}\cap X\cap\overline Y
\]
are smooth over $k$ of dimensions $n-m$ and $n-m-1$ respectively.

For $m=0$ we set $T_3=\bx$, $F=T=T'=\emptyset$ and apply Lemma~\ref{lm:Poonen}. We get a hypersurface $H\subset\P_k^{N_1}$ with an equation $\sigma_0\in H^0(\P_k^{N_1},\cO(r))$. We can view it as an element of $H^0(\P_k^{N_r},\cO(1))$. We have $\sigma_0(x)\ne0$ for $x\in\bx$ and $\dim(\{\sigma_0=0\}\cap T_i)<n-i$.

Assume that $\sigma_0$, \ldots, $\sigma_{m-1}$ are already constructed. To construct $\sigma_m$ we apply Lemma~\ref{lm:Poonen} to $\{\sigma_0=\dotsb=\sigma_{m-1}=0\}\cap T_i$, $\{\sigma_1=\dotsb=\sigma_{m-1}=0\}\cap T_i$, $T'=\{\sigma_1=\dotsb=\sigma_{m-1}=0\}\cap X\cap\overline Y$, $T=\{\sigma_1=\dotsb=\sigma_{m-1}=0\}\cap X$, and $F=\bx$.

Note that if at some step the subscheme $T_4:=\{\sigma_1=\dotsb=\sigma_{m-1}=0\}\cap T_1$ is finite, then we cannot apply the lemma because it might happen that $T_4\cap\bx\ne\emptyset$. However, in this case condition~\eqref{eq:DimEst} is automatic when $i=1$, so we apply Lemma~\ref{lm:Poonen} to $\{\sigma_0=\dotsb=\sigma_{m-1}=0\}\cap T_i$ and $\{\sigma_1=\dotsb=\sigma_{m-1}=0\}\cap T_2$.

By construction, $\sigma_0$, \ldots, $\sigma_{n-1}$ satisfy the conditions of the proposition. (Note that $\psi^{-1}(\psi(x))$ is an open subscheme of $\{\sigma_1=\ldots=\sigma_{n-1}=0\}$.)
\end{proof}

\begin{proof}[Proof of Proposition~\ref{pr:GoodNeighb} in the case, when $k$ is infinite.]
We will take $r=2$ and $l_1=\dotsb=l_{n-1}=1$. Set $V:=H^0(\P_k^{N_2},\cO(1))$; we view the vector space $V$ as a scheme over $k$.

\begin{lemma}\label{lm:AlgClosed}
Assume that $k$ is algebraically closed. Let $\overline X$, $X$, $\overline Y$, $T_1$, $T_2$ be as in the statement of Proposition~\ref{pr:GoodNeighb}. Let $x$ be a closed point of $X$. Then there is a non-empty open subset $W\subset V^n$ such that every point $(\sigma_0,\dotsc,\sigma_{n-1})\in W$ satisfies the conditions of the proposition with $\bx=\{x\}$. More precisely, condition~\eqref{GoodNeighb_Transversal} means that $\psi^{-1}(\psi(x))$ intersect $X$ and $X\cap\overline Y$ transversally.
\end{lemma}
\begin{proof}
Recall that we have a 2-fold Veronese embedding $\P_k^{N_1}\hookrightarrow\P_k^{N_2}$. Let $\tilde x\in V^\vee$ be a lift of $x$. Let $\Gr_x(N_2+1,n)$ stand for the Grassmannian of codimension $n-1$ projective subspaces containing $x$ in $\P^{N_2}_k$, or, equivalently, of codimension $n$ linear subspaces in $V^\vee$ containing the line $k\tilde x$. It follows from~\cite[Exp.~XI, Thm.~2.1(ii)]{SGA4-3} that there is a non-empty open subset $U\subset\Gr_x(N_2+1,n)$ such that every subspace from $U$ intersects $X$ and $X\cap\overline Y$ transversally.

Let us identify $V^n$ with $\Hom(V^\vee,k^n)$. Let $W'\subset V^n$ be the open subspace defined by the conditions that for all $\sigma=(\sigma_0,\dotsc,\sigma_{n-1})\in W'$ we have $\sigma_0(x)\ne0$, and $\sigma$ is of maximal rank $n$. Then $\Ker\sigma$ is a codimension $n+1$ vector subbundle in the trivial vector bundle $W'\times\A_k^{N_2+1}\to W'$ and $k\tilde x+\Ker\sigma$ is a codimension $n$ subbundle. Thus, $k\tilde x+\Ker\sigma$ gives a surjective morphism $\pi:W'\to\Gr_x(N_2+1,n)$ sending $(\sigma_0,\dotsc,\sigma_{n-1})\in W'$ to the Zariski closure of $\psi^{-1}(\psi(x))$. Set $W'':=\pi^{-1}(U)$, then $W''$ is open and non-empty. By construction, for each $(\sigma_0,\dotsc,\sigma_{n-1})\in W''$ we have $\sigma_0(x)\ne0$ and the intersections $\psi^{-1}(\psi(x))\cap X$ and $\psi^{-1}(\psi(x))\cap X\cap\overline Y$ are transversal. For every other condition of the proposition there is also a non-empty open subset in $V^n$ whose points possess the property (use an inductive argument similar to the proof in the finite field case for properties~\eqref{iii}--\eqref{v}). The intersection of these open subsets is non-empty, since $V^n$ is irreducible; it is the required set $W$.
\end{proof}

We return to the proof of Proposition~\ref{pr:GoodNeighb} in the case of infinite field. Let $\bar k$ be an algebraic closure of $k$. Consider the finite scheme $\bx\times_k\Spec\bar k$. Let $x_1,\dotsc,x_m\in\overline X_{\bar k}$ be all its closed points. Applying the previous lemma to $\overline X_{\bar k}$, $X_{\bar k}$, $\overline Y_{\bar k}$, $(T_1)_{\bar k}$, $(T_2)_{\bar k}$, and $x_i$, we get a dense open subset $\widetilde W_i\subset\widetilde V^n$, where $\widetilde V:=H^0(\P^{N_2}_{\bar k},\cO(1))$. Let $p\colon\widetilde V^n\to V^n$ be the projection. There is a non-empty open subset $W\subset V^n$ such that $p^{-1}(W)\subset\cap_{i=1}^m\widetilde W_i$. Since $k$ is infinite, we can find a $k$-rational point $(\sigma_0,\dotsc,\sigma_{n-1})\in W$. We claim that this point satisfies the conditions of the proposition. All conditions except~\eqref{GoodNeighb_Transversal} are clear.

Let us check condition~\eqref{GoodNeighb_Transversal}. Take $x\in\bx$ and set $s:=\psi(x)$.
We need to check that the scheme theoretic intersection $X\cap\psi^{-1}(s)$ is smooth and 1-dimensional over $s$. Let $x_i\in\bx\times_k\Spec\bar k$ be a point lying over $x$. We have an fpqc cover $x_i\to x\times_k\Spec\bar k\to s\times_k\Spec\bar k\to s$. The pullback of $X\cap\psi^{-1}(s)$ via this fpqc cover is smooth over $x$ and one-dimensional by construction and Lemma~\ref{lm:AlgClosed}. The statement follows by the fpqc descent. We prove similarly that $X\cap\overline Y\cap\psi^{-1}(s)$ is \'etale over $s$. This completes the proof of Proposition~\ref{pr:GoodNeighb}.
\end{proof}

\begin{proof}[Proof of Proposition~\ref{pr:Fibration}]
We may assume that $X$ is a closed subscheme of $\A^{N_1}_k$ for some $N_1>0$. Let~$\overline X$ be the closure of $X$ in $\P^{N_1}_k$. Set $Y:=\{f=0\}$ and let $\overline Y$ be the closure of $Y$ in $\overline X$. Note that $Y$ is of pure dimension $n-1$. Let $\overline Z$ be the closure of $Z$ in $\overline X$.

Apply Proposition~\ref{pr:GoodNeighb} with $T_1:=\overline Y\cup\overline Z$, and $T_2:=(\overline Y-Y)\cup(\overline Z-Z)$. Let $\sigma_i$ be the sections provided by Proposition~\ref{pr:GoodNeighb}. These sections give a rational morphism to a weighted projective space $\psi\colon\overline X\dashrightarrow\P_k(1,l_1,\dotsc,l_{n-1})$. We identify the coordinate chart of $\P_k(1,l_1,\dotsc,l_{n-1})$ given by the condition that the first homogeneous coordinate is non-zero with $\A_k^{n-1}$ (this open subset was denoted by $\hat\A_k$ in~\cite[Sect.~3.6]{FedorovMixedChar}). By Proposition~\ref{pr:GoodNeighb}\eqref{i}, $\psi(\bx)\subset\A_k^{n-1}$.

By Proposition~\ref{pr:GoodNeighb}\eqref{v}, $\psi$ is defined on $\overline Y\cup\overline Z$, and $\psi|_{\overline Y\cup\overline Z}$ is a projective morphism. Using Proposition~\ref{pr:GoodNeighb}\eqref{iii} and semicontinuity of fiber dimensions (see~\cite[Thm.~13.1.3]{EGAIV-3}), we find a neighborhood $S$ of $\psi(\bx)$ in $\A_k^{n-1}$ such that $\psi^{-1}(S)\cap(\overline Y\cup\overline Z)$ is quasi-finite over $S$. Since a quasi-finite projective morphism is finite, we see that $\psi^{-1}(S)\cap(\overline Y\cup\overline Z)$ is finite over $S$. Using property~\eqref{iv} of Proposition~\ref{pr:GoodNeighb}, we may shrink $S$ so that we have
\[
    \psi^{-1}(S)\cap(\overline Y\cup\overline Z)=\psi^{-1}(S)\cap(Y\cup Z).
\]
Using the semicontinuity of fiber dimensions again, we may shrink $S$ so that the fibers of $\psi$ over $S$ are of pure dimension one. Now by~\cite[Thm.~23.1]{MatsumuraCommRingTh} both $\psi$ and $\psi|_Y$ are flat over $S$. Let $W$ be the set of points of $\psi^{-1}(S)$ where $\psi$ is smooth. Since $\psi$ is flat over $S$, $W$ is open. Similarly, the set $W'$ of points of $\psi^{-1}(S)\cap Y$ such that $\psi|_Y$ is smooth is open. Consider the following closed subsets of $Z$ and $Y$ respectively: $\psi^{-1}(S)\cap Z-W$ and $\psi^{-1}(S)\cap Y-W\cap W'$. Since $\psi^{-1}(S)\cap Z$ and $\psi^{-1}(S)\cap Y$ are finite over $S$, $\psi(\psi^{-1}(S)\cap Z-W)$ and $\psi(\psi^{-1}(S)\cap Y-W\cap W')$ are closed in $S$. By Proposition~\ref{pr:GoodNeighb}\eqref{GoodNeighb_Transversal}, $\psi(\bx)$ does not intersect these closed sets. Thus, shrinking $S$, we may assume that $\psi$ is smooth at the points of $\psi^{-1}(S)\cap Z$ and $\psi^{-1}(S)\cap Y$ and $\psi|_Y$ is \'etale at the points of $\psi^{-1}(S)\cap Y$. It remains to take $X'$ to be the set of points of $\psi^{-1}(S)$ where $\psi$ is smooth.
\end{proof}

\subsection{Nice triples}\label{sect:NiceTriples} The goal of this section is to prove Proposition~\ref{pr:NiceTriple} formulated below. The data
\[
    (Y\subset Z\subset C\to\Spec\cO_{X,\bx},s)
\]
constructed below is similar to \emph{nice triples\/} in~\cite{PaninStavrovaVavilov,PaninNiceTriples,FedorovMixedChar}.

\begin{proposition}\label{pr:NiceTriple} Let $X$ be an integral affine scheme smooth over a field $k$, let $\bx\subset X$ be a non-empty finite set of closed points, and set $R:=\cO_{X,\bx}$. Let $\bF$ be a presheaf of pointed sets on $\Sm'/X$ satisfying property~(Lim) of Section~\ref{sect:GrSerre} and let $f\in k[X]$ be such that $f\ne0$ and the hypersurface $\{f=0\}$ is smooth over $k$. Consider an element
\[
    \cE\in\Ker\Bigl(\bF(R_f)\to\bF(K(X))\Bigr).
\]

Then there are:

\stepzero\noindstep a smooth affine $R$-scheme $C$ of pure relative dimension one;

\noindstep a section $s\in C(R)$;

\noindstep $R$-finite closed subschemes $Y\subset Z\subset C$ such that $Y$ is $R$-\'etale and $f|_{s^{-1}(Y)}=0$;

\noindstep\label{Cond:d} an essentially smooth morphism $\phi\colon C\to X$ such that $\phi\circ s\colon\Spec R\to X$ is the canonical morphism and $\phi^{-1}(\{f=0\})\subset Y$;

\noindstep\label{prop:NiceTriple4} an element
\[
    \cE'\in\Ker\Bigl(\bF(C-Y)\to\bF(C-Z)\Bigr)
\]
such that $\left(s|_{R_f}\right)^*\cE'=\cE$, where we view $C-Y$ as an $X$-scheme via $C-Y\hookrightarrow C\xrightarrow{\phi}X$.

Moreover, if $f=1$, then we may arrange the above so that $Y=\emptyset$.
\end{proposition}
Some comments are in order. The condition $f|_{s^{-1}(Y)}=0$ ensures that $s(\Spec R_f)\subset C-Y$. Thus the condition $\left(s|_{R_f}\right)^*\cE'=\cE$ makes sense.

Finally, note that $\Spec R$ has two structures of an $X$-scheme: the obvious one and the one obtained via the composition $\Spec R\xrightarrow{s}C\xrightarrow{\phi}X$. The condition on $\phi\circ s$ ensures that these two structures coincide.

\begin{proof}
    Applying property (Lim) of $\bF$ we may find an affine open subscheme $X''$ of $X$ containing~$\bx$, a closed subscheme $Z''\subsetneq X''$ and an object $\cE''\in\bF(X''_f)$ such that $\cE''|_{X''_f-Z''}=\star$ and $\cE''|_{R_f}=\cE$. In more detail, we have
    \[
        \Spec R_f=\lim\limits_{\longleftarrow}(X_\alpha)_f,
    \]
    where $X_\alpha$ is the inverse system of all affine neighborhoods of $\bx$.  By~(Lim)
    \[
        \bF(R_f)=\lim\limits_{\longrightarrow}\bF((X_\alpha)_f).
    \]
    Using explicit description of limits in the category of pointed sets, we find $\alpha$ and $\cE_\alpha\in\bF((X_\alpha)_f)$ such that $\cE_\alpha|_{R_f}=\cE$.

    Next, we have $\Spec K(X)=\lim\limits_{\longleftarrow}(X'_\beta)$, where $X'_\beta$ is the inverse system of all non-empty affine open subschemes of~$X_\alpha$. Using (Lim) again, we get
    \[
        \bF(K(X))=\lim_{\longrightarrow}\bF((X'_\beta)_f).
    \]
    Since $\cE_\alpha|_{K(X)}=\star$, using again explicit description of limits in the category of pointed sets, we see that $\cE_\alpha|_{(X'_\beta)_f}=\star$ for some $\beta$. Now we can take $X'':=X_\alpha$, $Z'':=(X_\alpha-X'_\beta)_{\mathrm{red}}$, and $\cE'':=\cE_\alpha$.

    Let $S\subset\A^{\dim X-1}_k$ and an $S$-scheme $X'$ be obtained by applying Proposition~\ref{pr:Fibration} to $X''$, $\bx$, $f|_{X''}$, and $Z''$. Set $C:=X'\times_S\Spec R$. Let $\can\colon\Spec R\to X'$ be the canonical morphism and $s:=\can\times\Id$ be the diagonal section.  Let $Y:=\{f|_{X'}=0\}\times_S\Spec R$ and
    \[
        Z:=(\{f|_{X'}=0\}\cup(Z''\cap X'))\times_S\Spec R,
    \]
    let $\phi$ be the composition of the projection $C\to X'$ and the open embedding $X'\hookrightarrow X$. Set $\cE':=(\phi|_{C-Y})^*\cE''$. By construction, if $f=1$, then $Y=\emptyset$. The desired conditions are satisfied by construction and by Proposition~\ref{pr:Fibration}.
\end{proof}

\subsection{Preparation to descent} In this section we improve the data constructed in Proposition~\ref{pr:NiceTriple} so that $\cE'$ can be descended to $\A^1_R$, where $R:={\cO_{X,\bx}}$.

\begin{proposition}\label{pr:NiceTripleImproved} Let $X$ be an integral affine scheme smooth over a field $k$, let $\bx\subset X$ be a non-empty finite set of closed points and set $R:=\cO_{X,\bx}$. Let $\bF$ be a presheaf of pointed sets on $\Sm'/X$ satisfying properties~(Lim) and~(LT) of Section~\ref{sect:GrSerre}. Assume that $f\in k[X]$ is such that $f\ne0$  and the hypersurface $\{f=0\}$ is smooth over $k$. Consider an element
\[
    \cE\in\Ker\Bigl(\bF(R_f)\to\bF(K(X))\Bigr).
\]
Then there are:

\stepzero\noindstep\label{prop:NiceTriple1} a smooth affine $R$-scheme $C$ of pure relative dimension one;

\noindstep a section $s\in C(R)$;

\noindstep $R$-finite closed subschemes $Y\subset Z\subset C$  such that $Y$ is $R$-\'etale and $f|_{s^{-1}(Y)}=0$;

\noindstep\label{prop:NiceTriple4'} an element
\[
    \cE'\in\Ker\Bigl(\bF(C-Y)\to\bF(C-Z)\Bigr)
\]
such that $\left(s|_{R_f}\right)^*\cE'=\cE$, where we view $C-Y$ as an $X$-scheme via the composition of morphisms $C-Y\hookrightarrow C\to\Spec R\to X$.

\noindstep\label{prop:NiceTriple5} An \'etale $R$-morphism $C\to\A^1_R$ such that $Z$ maps isomorphically onto a closed subscheme $Z'\subset\A^1_R$ and
\begin{equation}\label{eq:star}
    Z=Z'\times_{\A^1_R}C.
\end{equation}
Moreover, if $f=1$, then we may arrange the above so that $Y=\emptyset$.
\end{proposition}
Note that this time we view $C-Y$ as an $X$-scheme via a morphism different from that of Proposition~\ref{pr:NiceTriple}.
\begin{proof}
  Consider the data provided by Proposition~\ref{pr:NiceTriple}. If $f=1$, then we may assume that $Y=\emptyset$. We will ``improve'' this data to obtain the data required. Set $V:=s(\Spec R)$. Replacing $Z$ with $Z\cup V$, we may assume that $s$ factors through $Z$. Let~$Z_0$ be the connected component of $Z$ containing $V$. Replacing $C$ with its connected component containing $Z_0$, $Z$ with $Z_0$, and $Y$ with $Z_0\cap Y$, we may assume that $C$ is connected.

  Let $\by$ be the set of all closed points of $Z$. Then $\by$ is a non-empty finite set of closed points of $C$. Put $W:=\Spec\cO_{C,\by}$. Then $Z$ is a closed subscheme of $W$. So $Y$ is also a closed subscheme of $W$. Note that~$C$ has two essentially smooth morphisms to $X$: $\phi$ and the composition of the projection to $\Spec R$ with the canonical morphism $\Spec R\to X$; denote the latter morphism by $\phi_2$. Let $p_1,p_2\colon W\to X$ be the compositions of the canonical morphism $W\to C$ with $\phi$ and $\phi_2$ respectively. The condition~\eqref{Cond:d} of Proposition~\ref{pr:NiceTriple} implies that $p_1|_V=p_2|_V$, so we can use property (LT) of $\bF$ to find a finite \'etale $k$-morphism $\pi\colon W'\to W$, a section $\Delta\colon V\to W'$, and an isomorphism $(p_1\circ\pi)^*\bF\simeq(p_2\circ\pi)^*\bF$. Using this isomorphism, we can view
  \[
    \cE'':=(\pi|_{W'-W'\times_WY})^*(\cE'|_{W-Y})
  \]
  as an object of $\bF(W'-W'\times_WY)$, where $W'-W'\times_WY$ is viewed as an $X$-scheme via the composition
  \[
      W'-W'\times_WY\to W'\xrightarrow{\pi}W\to C\to\Spec R\to X.
  \]
  Set $s':=\Delta\circ s\colon\Spec R\to W'$. The condition $\left(s|_{R_f}\right)^*\cE'=\cE$ together with diagram~\eqref{eq:CD} tell us that $\left(s'|_{R_f}\right)^*\cE''=\cE$.

  We can find an affine open subscheme $C_1$ of $C$ containing $Z$ and a finite \'etale morphism $C'_1\to C_1$ such that $W'=W\times_{C_1}C'_1$. We can write $W=\lim\limits_{\longleftarrow}C_\alpha$, where $C_\alpha$ are open affine subschemes of $C_1$ containing $Z$. Then $W'=\lim\limits_{\longleftarrow}C'_\alpha$, where $C'_\alpha:=C_\alpha\times_{C_1}C'_1$.

  Set $Z':=Z\times_{C_1}C'_1$ and $Y':=Y\times_{C_1}C'_1$. Note that $Z'$ and $Y'$ are closed subschemes of each $C'_\alpha$ and that $Z'$ and $Y'$ are finite over $R$. Note also, that if $f=1$, then $Y$ is empty, so $Y'$ is empty as well.

  Using the property~(Lim) of $\bF$, we can find $\alpha$ and an element $\cE_\alpha\in\bF(C'_\alpha-Y')$ such that $\cE_\alpha|_{W'-Y'}=\cE''$ and $\cE_\alpha$ restricts to $\star$ on $C'_\alpha-Z'$ (this is similar to the beginning of the proof of Proposition~\ref{pr:NiceTriple}). Here we view $C'_\alpha-Y'$ as an $X$-scheme via the composition
  \[
    C'_\alpha-Y'\hookrightarrow C'_\alpha\to C_1\hookrightarrow C\to\Spec R\to X.
  \]

  Renaming $C'_\alpha\mapsto C$, $s'\mapsto s$, $Y'\mapsto Y$, $Z'\mapsto Z$, and $\cE_\alpha\mapsto\cE'$, we satisfy properties~\eqref{prop:NiceTriple1}--\eqref{prop:NiceTriple4'} of the proposition.

  To arrange property~\eqref{prop:NiceTriple5}, we argue as in the proof of~\cite[Prop.~6.5]{CesnaviciusGrSerre}.  Let us give more details. By construction, $s$ factors through $Z$. Applying~\cite[Lm.~6.1]{CesnaviciusGrSerre} to the closed subscheme $Z$ of $C$, we see that we can change the data so we can assume that for all $x\in\bx$ we have
  \[
    \#\{z\in Z_x\colon [k(z):k(x)]=d\}<\#\{z\in \A^1_x\colon [k(z):k(x)]=d\}\text{ for every }d\ge1.
  \]
  Now, applying~\cite[Lm.~6.3]{CesnaviciusGrSerre} with $Y=\emptyset$ and replacing $C$ with an appropriate affine Zariski neighborhood of $Z$, we get a flat morphism $C\to\A^1_R$ such that~$Z$ maps isomorphically onto a closed subscheme $Z'\subset\A^1_R$ and such that~\eqref{eq:star} is satisfied. This morphism is \'etale in a neighborhood of $Z$, so, shrinking $C$, we obtain the desired morphism.
\end{proof}

\subsection{Descent to the affine line}\label{sect:DescendA1} Recall that in Definition~\ref{def:ElemDistSqDescend} we defined the notion of a Nisnevich semisheaf of pointed sets.
\begin{proposition}\label{pr:A1Descend}
Let $X$ be an integral affine scheme smooth over a field $k$, let $\bx\subset X$ be a non-empty finite set of closed points and set $R:=\cO_{X,\bx}$. Let $\bF$ be a Nisnevich semisheaf of pointed sets on $\Sm'/X$ satisfying properties~(Lim) and~(LT) of Section~\ref{sect:GrSerre} and let $f\in k[X]$ be such that $f\ne0$ and the hypersurface $\{f=0\}$ is smooth over $k$. Consider an element
\[
    \cE\in\Ker\Bigl(\bF(R_f)\to\bF(K(X))\Bigr).
\]
Then there are:

\noindstep a section $s\in\A^1_R(R)$;

\noindstep $R$-finite closed subschemes $Y\subset Z\subset \A^1_R$ such that $Y$ is $R$-\'etale and $f|_{s^{-1}(Y)}=0$;

\noindstep an element
\[
    \cE'\in\Ker\Bigl(\bF(\A^1_R-Y)\to\bF(\A^1_R-Z)\Bigr)
\]
such that $\left(s|_{R_f}\right)^*\cE'=\cE$.

Moreover, if $f=1$, then we may arrange the above so that $Y=\emptyset$.
\end{proposition}
\begin{proof}
Consider the data $Y\subset Z\subset C\to\A_R^1$, $s$, $\cE'$, and $Z'$ provided by Proposition~\ref{pr:NiceTripleImproved}. If $f=1$, we may assume that $Y=\emptyset$. Since the morphism $C\to\A^1_R$ maps $Z$ isomorphically onto $Z'$, it maps $Y$ isomorphically onto a closed subscheme $Y'\subset Z'$. The square
\[
\begin{CD}
C-Z @>>> C-Y\\
@VVV @VVV \\
\A^1_R-Z' @>>>\A^1_R-Y'
\end{CD}
\]
is an elementary distinguished square as in Definition~\ref{def:ElemDistSq}. Since $\bF$ is a Nisnevich semisheaf, by Definition~\ref{def:ElemDistSqDescend},~$\cE'$ can be descended to an element $\cE''\in\bF(\A^1_R-Y')$ whose restriction to $\A^1_R-Z'$ is $\star$. Let $s'$ be the composition of $s$ and the morphism $C\to\A^1_R$.

It remains to rename $s'\mapsto s$, $Y'\mapsto Y$, $Z'\mapsto Z$, $\cE''\mapsto\cE'$.
\end{proof}

\subsection{Proof of Theorem~\ref{th:GrSerreFun}}\label{sect:EndPrfGrSerreFun}
We are now ready to finish the proof of Theorem~\ref{th:GrSerreFun}. We may assume that $\by$ is non-empty as otherwise $\cO_{X,\by}=K(X)$ and the statement is trivial. Let $X$, $\by$, and $\bF$ be as in the statement of the theorem. Consider
\[
    \cE\in\Ker\Bigl(\bF(\cO_{X,\by})\to\bF(K(X))\Bigr).
\]
Using the property (Lim) of $\bF$, we find an open affine neighborhood~$X'$ of $\by$ in $X$ and $\cE'\in\bF(X')$ such that $\cE'|_{\cO_{X,\by}}=\cE$ (cf.~beginning of the proof of Proposition~\ref{pr:NiceTriple}). For any $y\in\by$ choose a closed point $x\in X'$ such that $x$ is in the closure of $y$. Let $\bx$ be the set of these closed points and set $R:=\cO_{X,\bx}$. It is enough to show that $\cE'':=\cE'|_R=\star$ (because $\cE''|_{\cO_{X,\by}}=\cE$). Note that $\cE''|_{K(X)}=\cE|_{K(X)}=\star$.

Applying Proposition~\ref{pr:A1Descend} with $f=1$ to $\cE''$, we find

\stepzero\noindstep a section $s\in\A^1_{R}(R)$;

\noindstep an $R$-finite closed subscheme $Z\subset \A^1_R$;

\noindstep an element
\[
    \cE'''\in\Ker\Bigl(\bF(\A^1_R)\to\bF(\A^1_R-Z)\Bigr)
\]
such that $s^*\cE'''=\cE''$. It remains to use property~(Sec) of~$\bF$.
\qed

\subsection{Open subschemes of affine lines} The following proposition is crucial for the proof of Theorem~\ref{th:NisnevichFun}.
\begin{proposition}\label{pr:A1}
Let $X$ be a smooth integral affine $k$-scheme, where $k$ is an infinite field. Let $\bF$ be a Nisnevich semisheaf of pointed sets on $\Sm'/X$ satisfying properties (Lim), (LT), (SecF) of Section~\ref{sect:GrSerre} and property (A1F) of Section~\ref{sect:NisConjSemisheaves}. Let $\bx$ be a non-empty finite set of closed points of $X$ and put $R:=\cO_{X,\bx}$. Let $Y\subset Z\subset\A^1_R$ be closed subschemes finite over $R$ such that $Y$ is \'etale over $R$. Then the map $\bF(\A^1_R-Y)\to\bF(\A^1_R-Z)$ has a trivial kernel.
\end{proposition}

We will use the following simple lemma.
\begin{lemma}\label{lm:EtaleSec}
Let $\phi\colon\widetilde T\to T$ be an \'etale separated morphism of Noetherian schemes. Let $\Delta\colon T\to\widetilde T$ be a section of $\phi$. Then $\widetilde T=\Delta(T)\sqcup T'$, where $T'$ is a closed subscheme of $\widetilde T$ disjoint from $\Delta(T)$.
\end{lemma}
\begin{proof}
Since $\phi$ is separated, $\Delta$ is a closed embedding. On the other hand, by~\cite[Cor.~17.3.5]{EGAIV.4} the morphism $\Delta\colon T\to\widetilde T$ is \'etale. Thus, it is open. Hence, $\Delta(T)$ is both a closed and an open subscheme of $\widetilde T$.
\end{proof}

\begin{proof}[Proof of Proposition~\ref{pr:A1}]
Let
\[
    \cE\in\Ker\Bigl(\bF(\A^1_R-Y)\to\bF(\A^1_R-Z)\Bigr).
\]
We need to show that $\cE=\star$.

\emph{Case 1. Assume that $Y=Y_0\times_k\Spec R$, where $Y_0$ is a closed subscheme of $\A^1_k$.} Since
\[
    \A^1_R-Y=(\A^1_k-Y_0)\times_k\Spec R,
\]
and the scheme $\A^1_k-Y_0$ is affine and $k$-smooth, we can apply Corollary~\ref{cor:Family}. Thus we only need to check that $\cE_K:=\cE|_{\A^1_K-Y_K}=\star$, where $K$ is the fraction field of $R$. Note that $Y_K$ and $Z_K$ are finite subschemes of $\A^1_K$. Let $T_K$ denote the set theoretic difference $Z_K-Y_K$. Recall from Remark~\ref{rem:Zariski}(i) that $\bF$ satisfies gluing in the Zariski topology. Since $\cE_K|_{\A^1_K-Z_K}=\star$, we can glue $\cE_K$ with $\star\in\bF(\A^1_K-T_K)$ to obtain an element $\cE'\in\bF(\A^1_K)$. It follows from $(A1F)$ that $\cE'=\star$. Since $\cE'|_{\A^1_K-Y_K}=\cE_K$, we are done.

\emph{Case 2.  Assume that we have an integer $d\in\Z_{\ge0}$ such that $Y=\bigsqcup_{i=1}^dY_i$, where for each $i$ the restriction of the projection $\A^1_R\to\Spec R$ to $Y_i$ is an isomorphism.}

Enlarging $Z$, we may assume that
\begin{equation}\label{eq:01inZ}
  (0\times\Spec R)\sqcup(1\times\Spec R)\subset Z.
\end{equation}
We need a lemma.
\begin{lemma}
     There is a Zariski neighborhood $W$ of $Z$ in $\A^1_R$ and an \'etale $R$-morphism $\phi\colon W\to\A^1_R$ such that $\phi$ maps $Z$ isomorphically onto a closed subscheme $Z'\subset\A^1_R$ such that $\phi^{-1}(Z')=Z$, and $\phi(Y)$ is of the form $Y_0\times_k\Spec R$.
\end{lemma}
\begin{proof}
Let $t_1,\ldots,t_d$ be distinct elements of $k$ (we are using that $k$ is infinite here). Note that for any extension $k'$ of $k$, the elements $t_i$ gives rise to a $k'$-point of $\A^1_{k'}$. We abuse notation denoting this point by $t_i$.

We claim that there is a closed $R$-embedding $\iota\colon Z_{(2)}\hookrightarrow\A^1_R$, where $Z_{(2)}$ is the first infinitesimal neighborhood of $Z$ in $\A^1_R$, such that for all $i$ we have $\iota(Y_i)=t_i\times_k\Spec R$. Indeed, take a point $x\in\bx$ and consider the finite scheme $(Z_{(2)})_x\subset\A^1_x$. We can write
\[
    (Z_{(2)})_x=\bigsqcup_{i=1}^nZ_{x,i},
\]
where the underlying set of $Z_{x,i}$ consists of a single point $z_{x,i}$. By reordering the schemes $Z_{x,i}$, we may assume that $(Y_i)_x=z_{x,i}$ for $i=1,\ldots,d$. Consider a bijection $\alpha\colon k(x)\to k(x)$ such that for $i=1,\ldots,d$ we have $\alpha((Y_i)_x)=t_i$ (here $k(x)$ is the residue field of $x\in\Spec R$).

We define the morphism $\iota_x\colon(Z_{(2)})_x\to\A^1_x$ as follows. If $z_{x,i}$ is not $k(x)$-rational, then the restriction of~$\iota_x$ to $Z_{x,i}$ is the original embedding. Otherwise, it is the composition of this embedding with the shift by $-z_{x,i}+\alpha(z_{x,i})$. Note that the restriction of $\iota_x$ to any $Z_{x,i}$ is a closed embedding. Since $\iota_x$ is injective on the closed points of $(Z_{(2)})_x$, it is a closed embedding. By construction $\iota_x((Y_i)_x)=t_i$. An $R$-morphism from any $R$-scheme to $\A^1_R$ is given by a global function. Thus, by a version of the Chinese Remainder Theorem, we can extend $\coprod_{x\in\bx}\iota_x$ to an $R$-morphism $\iota$ such that for all $i$ we have $\iota(Y_i)=t_i\times_k\Spec R$. Now it follows from Nakayama's Lemma that $\iota$ is a closed embedding.

Next, we may extend $\iota$ to $\phi\colon\A^1_R\to\A^1_R$. By~\eqref{eq:01inZ}, this morphism is non-constant on the closed fibers $\A^1_x$ for $x\in\bx$, so, by semicontinuity of fiber dimensions (see~\cite[Thm.~13.1.3]{EGAIV-3}), it is quasi-finite. Since $\A^1_R$ is regular, $\phi$ is flat (see~\cite[Thm.~23.1]{MatsumuraCommRingTh}). Since $\phi|_{Z_{(2)}}$ is a closed embedding, $\phi$ is \'etale along $Z_{(2)}$. Thus it is \'etale on a neighborhood $W'$ of $Z$. Next, $\phi$ induces an isomorphism $Z\to\phi(Z)$, so it has a section over $\phi(Z)$. Since $\phi|_{W'}$ is \'etale, by Lemma~\ref{lm:EtaleSec}, we get $\phi^{-1}(\phi(Z))\cap W'=Z\sqcup Z_1$ for a closed subscheme $Z_1$ of $W'$, and we take $W:=W'-Z_1$. Finally, we have
\[
    \phi(Y)=\{t_1,\ldots,t_d\}\times_k\Spec R.
\]
The proof is complete.
\end{proof}
We return to the proof of Case~2 of Proposition~\ref{pr:A1}. The square
\[
\begin{CD}
W-Z @>>> W-Y\\
@VVV @VVV \\
\A^1_R-Z' @>>>\A^1_R-Y_0\times_k\Spec R
\end{CD}
\]
is an elementary distinguished square. Since $\bF$ is a Nisnevich semisheaf, $\cE|_{W-Y}$ can be descended to an element $\cE'\in\bF(\A^1_R-Y_0\times_k\Spec R)$ whose restriction to $\A^1_R-Z'$ is $\star$. By the previous case, we have $\cE'=\star$, so $\cE|_{W-Y}=\star$. Consider the Zariski cover
\[
    \A^1_R=(\A^1_R-Y)\cup W.
\]
We can glue $\cE$ with $\star$ to obtain an element $\cE''\in\bF(\A^1_R)$ such that $\cE''|_{\A^1_R-Z}=\star$. Using the previous case with $Y=\emptyset$, we see that $\cE''=\star$. Thus, $\cE=\cE''|_{\A^1_R-Y}=\star$.

\emph{Case 3. General case.} For an $R$-quasi-finite scheme $T$, let $\deg_RT$ denote the length of the generic fiber of $T\to\Spec R$, and let $d(T/\Spec R)$ be the difference between $\deg_RT$ and the number of connected components of $T$. In other words, for a connected and quasi-finite over $R$ scheme $T$ we set $d(T/\Spec R)=\deg_RT-1$, and then we extend this definition to disconnected schemes by additivity.

We induct on $d(Y/\Spec R)$. If $d(Y/\Spec R)=0$, then all connected components of $Y$ are of degree one over $\Spec R$, so we are in the previous case. Assume that the statement is proved for smaller values of $d(Y/\Spec R)$. Let $V$ be a connected component of $Y$ of degree at least two over $R$. Set $Y':=Y\times_R V$, $Z':=Z\times_R V$, $\cE':=\cE|_{\A^1_V-Y'}$. Let $\Delta\subset V\times_R V\subset Y'$ be the diagonal component. Since $V\times_RV$ decomposes into the union of at least two connected components (one being $\Delta$), we have $d(Y'/V)<d(Y/\Spec R)$. Since $\cE'|_{\A^1_V-Z'}=\star$, by the induction hypothesis we have $\cE'=\star$. In more detail, we can find an \'etale morphism $\pi\colon X'\to X$ such that $V\simeq\cO_{X',\pi^{-1}(\bx)}$. Then we can apply the induction hypothesis to $R':=\cO_{X',\pi^{-1}(\bx)}$, $Y'$, and $Z'$.

The \'etale morphism $\phi\colon\A^1_V\to\A^1_R$ maps $\Delta$ isomorphically onto $V$. Let $W$ be a Zariski neighborhood of $\Delta$ in $\A^1_V$ such that $\phi^{-1}(V)\cap W=\Delta$ (existing by Lemma~\ref{lm:EtaleSec}).
Then the square
\[
\begin{CD}
W-Y' @>>> W-(Y'-\Delta)\\
@VVV @VVV \\
\A^1_R-Y @>>>\A^1_R-(Y-V)
\end{CD}
\]
is an elementary distinguished square. Note that
\[
    \cE|_{W-Y'}=\cE'|_{W-Y'}=\star|_{W-Y'}=\star.
\]
Thus, since $\bF$ is a Nisnevich semisheaf, we can glue $\cE$ with $\star\in\bF(W-(Y'-\Delta))$ to obtain an element $\cE'''\in\bF(\A^1_R-(Y-V))$ whose restriction to $\A^1_R-Y$ is $\cE$. Since $\deg_RV\ge2$, we have $d((Y-V)/\Spec R)<d(Y/\Spec R)$, so we can apply the induction hypothesis to conclude that $\cE'''=\star$. Thus $\cE=\star$.
\end{proof}

\subsection{Proof of Theorem~\ref{th:NisnevichFun}}
We are now ready to finish the proof of Theorem~\ref{th:NisnevichFun}. Consider
\[
    \cE\in\Ker\Bigl(\bF((\cO_{X,\by})_f)\to\bF(K(X))\Bigr).
\]
We may assume that $\by$ is non-empty as otherwise $(\cO_{X,\by})_f=K(X)$ and the statement is trivial. Using property (Lim) of $\bF$, we find an open affine neighborhood $X'$ of $\by$ in $X$ and $\cE'\in\bF(X'_f)$ such that $\cE'|_{(\cO_{X,\by})_f}=\cE$ (this is similar to the beginning of the proof of Proposition~\ref{pr:NiceTriple}).

For any $y\in\by$ choose a closed point $x\in X'$ such that $x$ is in the closure of $y$. Let~$\bx$ be the set of these closed points. Set $R:=\cO_{X,\bx}$. It is enough to show that $\cE'':=\cE'|_{R_f}=\star$ (because $\cE''|_{(\cO_{X,\by})_f}=\cE$). Note that $\cE''|_{K(X)}=\star$.

Applying Proposition~\ref{pr:A1Descend} we find

\stepzero\noindstep a section $s\in\A^1_R(R)$;

\noindstep $R$-finite closed subschemes $Y\subset Z\subset \A^1_R$ such that $Y$ is $R$-\'etale and $f|_{s^{-1}(Y)}=0$;

\noindstep an element
\[
    \cE'''\in\Ker\Bigl(\bF(\A^1_R-Y)\to\bF(\A^1_R-Z)\Bigr)
\]
such that $(s|_{R_f})^*\cE'''=\cE''$. It remains to use Proposition~\ref{pr:A1}.
\qed

\section{Applications to torsors}
For completeness we start with deriving the (already known) conjecture of Grothendieck and Serre for torsors over regular semilocal integral domains containing fields from Theorem~\ref{th:GrSerreFun}. Then we derive Theorem~\ref{th:Nisnevich} from Theorem~\ref{th:NisnevichFun}.
\subsection{Re-proving the Grothendieck--Serre conjecture} Let $R$ be a regular semilocal integral domain containing a field $k$ and let $\bG$ be a reductive group scheme over $R$. We want to show that a generically trivial $\bG$-torsor $\cE$ over $R$ is trivial. We may assume that $k$ is the prime field of $R$ so, in particular, that $k$ is perfect. Now, by Popescu's Theorem (see~\cite{Popescu},\cite{SwanOnPopescu},\cite{SpivakovskyPopescu}), we can write $R=\lim\limits_{\longrightarrow} R_\alpha$, where each $R_\alpha$ is smooth (and, in particular, of finite type) over~$k$. Changing the direct system $R_\alpha$, we may assume that the rings $R_\alpha$ are integral. For some $\alpha$ we can find a reductive group scheme $\bG_\alpha$ over $R_\alpha$ such that $\bG_\alpha|_R\simeq\bG$ and a $\bG_\alpha$-torsor $\cE_\alpha$ over $R_\alpha$ such that $\cE_\alpha|_R\simeq\cE$ and $\cE_\alpha|_{K(R_\alpha)}$ is trivial.

Set $X:=\Spec R_\alpha$ and let $\by\subset X$ be the image of the set of all closed points of $\Spec R$ in $X$. Then $\Spec R\to X$ factors through $\Spec\cO_{X,\by}$. Set $\cE':=\cE_\alpha|_{\cO_{X,\by}}$. Since $\cE'|_R=\cE$, it is enough to show that~$\cE'$ is trivial.

Consider the functor $\bF$ on $\Sm'/X$ given by $\bF(Y):=H^1(Y,\bG_\alpha)$. It is enough to show that the map $\bF(\cO_{X,\by})\to\bF(K(X))$ has a trivial kernel. This will follow from Theorem~\ref{th:GrSerreFun} as soon as we verify the hypothesis.

Note that torsors satisfy the \'etale descent. It follows that $\bF$ is a Nisnevich semisheaf (cf.~Definition~\ref{def:ElemDistSqDescend} and the first part of the proof of~\cite[Ch.~3, Prop.~1.4]{MorelVoevodsky}). In more detail, consider a distinguished square~\eqref{eq:ElemDistSq} with $W=U\times_XV$ and $\bG$-torsors $\cE_U$ and $\cE_V$ on $U$ and $V$ respectively whose restrictions to $W$ are isomorphic. Choose an isomorphism $\phi\colon\cE_U|_W\to\cE_V|_W$. Note that $U\sqcup V$ is an \'etale cover of $Y$. To descend $\cE_U\sqcup\cE_V$ to $X$, we need to construct an isomorphism between the two pullbacks of $\cE_U\sqcup\cE_V$ to
\[
    (U\sqcup V)\times_Y(U\sqcup V)=(U\times_YU)\sqcup(U\times_YV)\sqcup(V\times_YU)\sqcup(V\times_YV)
\]
satisfying the cocycle condition. The isomorphism on $U\times_YU=U$ is the identity. The isomorphisms on $U\times_YV$ and on $V\times_YU$ are given by $\phi$. To construct an isomorphism on $V\times_YV$, we consider the Zariski covering
\[
    V\times_YV=\Delta(V)\sqcup(U\times_YV\times_YV),
\]
where $\Delta\colon V\to V\times_UV$ is the diagonal embedding. The two compositions
\[
    V\xrightarrow{\Delta} V\times_YV\to (U\sqcup V)\times_Y(U\sqcup V)\rightrightarrows U\sqcup V
\]
are equal so the two pullbacks of $\cE_U\sqcup\cE_V$ to $\Delta(V)$ are canonically isomorphic. Next, the two morphisms from
$U\times_YV\times_YV$ to $U\sqcup V$ can be factored through
\[
    U\times_YV\times_YV=(U\times_YV)\times_U(U\times_YV)\rightrightarrows U\times_YV,
\]
so the two pullbacks of $\cE_V$ are identified with the help of $\phi$ with the two pullback of $\cE_U$ via
\[
    (U\times_YV)\times_U(U\times_YV)\rightrightarrows U\times_YV\to U,
\]
and again they are canonically isomorphic. It is easy to see that the isomorphisms agree on the intersection $\Delta(V)\cap(U\times_YV\times_YV)$. One checks that the isomorphisms satisfy the cocycle condition. We see that $\bF$ is a Nisnevich semisheaf.

It is clear that~$\bF$ satisfies property~(Lim). Property~(LT) follows from~\cite[Thm.~3.1]{PaninNiceTriples}. Property~(Sec) follows from~\cite[Thm.~4]{FedorovGrSerreNonSC}.

\subsection{Proof of Theorem~\ref{th:Nisnevich}}\label{sect:ProofNis}
We may assume that $f\ne0$ as otherwise $R_f=0$ and the statement is trivial. Let $\fm_1$, \ldots, $\fm_n$ be all the maximal ideals of $R$. Let $k'$ be the prime field of~$R$. Since $k'$ is perfect, by Popescu's Theorem (see~\cite{Popescu},\cite{SwanOnPopescu},\cite{SpivakovskyPopescu}) we can write $R=\lim\limits_{\longrightarrow} R_\alpha$, where the rings $R_\alpha$ are smooth over $k'$. We may assume that these rings are integral. Denote the canonical morphism $R_\alpha\to R$ by $\psi_\alpha$.

Then for some $\alpha$ we can find an $R_\alpha$-group scheme $\bG_\alpha$, an element $f_\alpha\in R_\alpha$, and a $\bG_\alpha$-torsor $\cE_\alpha$ over $(R_\alpha)_{f_\alpha}$ such that $(\bG_\alpha)^{ad}$ is strongly locally isotropic, $\bG_\alpha|_R=\bG$, $\psi_\alpha(f_\alpha)=f$, and $\cE_\alpha|_{R_f}\simeq\cE$.

By assumption, for all $i$ we have $f\notin\fm_i^2$. Thus, for all $i$ we have $f_\alpha\notin\psi_\alpha^{-1}(\fm_i)^2$. Since $k'$ is perfect, this implies that $\Spec(R_\alpha/f_\alpha R_\alpha)$ is $k'$-smooth in a neighborhood of $\bigcup_{i=1}^n\psi_\alpha^{-1}(\fm_i)$. That is, there is $g\in R_\alpha$ such that for all $i$ we have $g\notin\fm_i$ and $R'/f'R'$ is $k'$-smooth, where $R'=(R_\alpha)_g$ and $f'$ is the image of $f_\alpha$ in $R'$. Since $R$ is semilocal, $\psi_\alpha$ factors as $R_\alpha\xrightarrow{\psi'}R'\to R$. Note that $R'$ is smooth and of finite type over $k'$.

The morphism $\psi'$ decomposes as
\[
    R'\to R'\otimes_{k'}k\xrightarrow{\phi}R.
\]
The induced morphism $\Spec R\to\Spec(R'\otimes_{k'}k)$ factors as
\[
    \Spec R\xrightarrow{\phi^*}X\to\Spec(R'\otimes_{k'}k),
\]
where $X$ is a connected component of $\Spec(R'\otimes_{k'}k)$. Set $\bG':=\bG_\alpha|_X$. Then $(\bG')^{ad}$ is strongly locally isotropic.

Consider the functor $\bF$ on $\Sm'/X$ given by $\bF(Y):=H^1(Y,\bG')$. We have explained above that $\bF$ is a Nisnevich semisheaf satisfying properties~(Lim) and~(LT) of Section~\ref{sect:GrSerre}. Let us verify that it satisfies properties  (SecF), and (A1F). Note that $X$ is normal, so $\bG$ can be embedded into $\GL_{n,X}$ for some $n$, see~\cite[Cor.~3.2]{ThomasonResolution}. Now property~(SecF) follows from~\cite[Thm.~5]{FedorovGrSerreNonSC}. Finally, (A1F) follows from~\cite[Cor.~3.10(a)]{GilleTorseurs}.

Set $\by:=\phi^*(\{\fm_1,\ldots,\fm_n\})\subset X$.

Set $f_X:=(f'\otimes_{k'}1)|_X\in k[X]$. Then $\{f_X=0\}$ is the base change of $\Spec(R'/f'R')$ via $\Spec k\to\Spec k'$. Thus, this hypersurface is smooth. Let $\cE_X$ be the restriction of $\cE_\alpha$ to $(\cO_{X,\by})_{f_X}$. It follows from Theorem~\ref{th:NisnevichFun} that $\cE_X$ is trivial. However, $\cE_X|_{R_f}\simeq\cE$ and we are done. \qed

\subsection{Proof of Theorem~\ref{th:NisnevichAbelian}}\label{sect:ProofAbelian} We may assume that $f\ne0$ as otherwise $R_f=0$ and the statement is trivial. It is enough to show that the homomorphism $H^n(R_f,\bG)\to H^n(K,\bG)$ has trivial kernel. Let $k$ be a field of characteristic zero contained in $R$. Since $k$ is perfect, we can apply Popescu's Theorem to write $R=\lim\limits_{\longrightarrow} R_\alpha$, where each $R_\alpha$ is smooth (and, in particular, of finite type) over~$k$. Thus, using~\cite[Exp.~VII, Cor.~5.9]{SGA4-2} we reduce to the case when $R$ is a semi-local ring of finitely many points on a smooth $k$-scheme $X$ and $\bG$ is a restriction of a group scheme $\bG'$ of multiplicative type defined over $X$. Consider the functor $\bF$ on $\Sm'/X$ given by the \'etale cohomology group $\bF(Y):=H^n(Y,\bG'|_Y)$.

It is enough to show that $\bF$ satisfies the conditions of Theorem~\ref{th:NisnevichFun}. Let us show first that $\bF$ is a Nisnevich semisheaf. Consider an elementary distinguished square~\eqref{eq:ElemDistSq} and set $Z:=(Y-U)_{\mathrm{red}}$, $Z':=(V-W)_{\mathrm{red}}$. By~\cite[Ch.~III, Prop.~1.27]{MilneEtale} the natural morphisms $H_Z^p(Y,\bG')\to H_{Z'}^p(V,\bG'|_V)$ are isomorphisms for all $p\ge0$. Now the statement follows via a diagram chase from the exact sequence~\cite[Ch.~III, Prop.~1.25]{MilneEtale}.

Property~(Lim) follows from~\cite[Exp.~VII, Cor.~5.9]{SGA4-2}. Property~(LT) follows from~\cite[Thm.~3.1]{PaninNiceTriples}. Properties~(Sec) and~(A1F) follow from the homotopy invariance of \'etale cohomology, which is shown in the proof of~\cite[Thm.~4.2]{ColliotTeleneOjanguren}. \qed

\section{Proof of Theorem~\ref{th:CounterExample}}\label{sect:CounterEx}
Theorem~\ref{th:CounterExample} will be derived from the following proposition.

\begin{proposition}\label{pr:CounterExample}
  Let $B$ be an integral Noetherian normal local ring that is not a field. Assume that $B$ contains an infinite field. Let $\bG$ be a nontrivial semisimple $B$-group scheme anisotropic over the fraction field of $B$. Then there is a maximal ideal $\fm\subset B[t]$ and a generically trivial $\bG$-torsor $\cE$ over $\A^1_{B[t]_\fm}$ that cannot be extended to $\P^1_{B[t]_\fm}$.
\end{proposition}

\begin{remark}
  It is likely that the assumption that $B$ contains a field is not necessary. It is probably also enough to assume that $\bG$ is reductive. However, we want to use the results of~\cite{FedorovExotic}, where these assumptions were made.
\end{remark}

\begin{proof}[Derivation of Theorem~\ref{th:CounterExample} from Proposition~\ref{pr:CounterExample}]
  In~\cite[Example~2.4]{FedorovExotic} for any algebraically closed field $k$ we constructed a connected smooth variety $X$ over $k$ and a $\Spin_7$-torsor $\cF$ over $X$ such that~$\cF$ cannot be reduced to a proper parabolic subgroup of $\Spin_7$ over the generic point of $X$.

  Recall that a semisimple group scheme $\bG$ over a semilocal scheme $S$ is anisotropic, if and only if it contains no proper parabolic subgroup schemes (see~\cite[Exp.~XXVI, Cor.~6.14]{SGA3-3}). Let $\bG:=\Aut(\cF)$ be the group scheme of automorphisms of $\cF$ as a $\Spin_7$-torsor. Then $\bG$ is a simple simply-connected group scheme (a strongly inner form of $\Spin_7$) such that its generic fiber is anisotropic by~\cite[Prop.~3.2]{FedorovExotic}. Since $k$ is algebraically closed, $X$ is not a point. Let $y\in X$ be a closed point. Then by Proposition~\ref{pr:CounterExample} applied to $B:=\cO_{X,y}$ there is a generically trivial $\bG$-torsor over $\A^1_R$, where $R$ is the local ring of a point on $\A^1_B$, that cannot be extended to $\P^1_R$.
\end{proof}

Recall that in~\cite[Sect.~5]{FedorovExotic} for a connected affine scheme $U$ and a semisimple $U$-group scheme~$\bG$ we defined the affine Grassmannian $\Gr_\bG$. This is an ind-projective ind-scheme over $U$. We will use these results here.

\subsection{Proof of Proposition~\ref{pr:CounterExample}}
The following proposition is crucial for the proof. The proposition asserts that, in the anisotropic situation, the possibility of extending a torsor over $\A^1_k\times X$ to $\P^1_k\times X$ can be verified locally over $X$.
\begin{proposition}\label{pr:ExtAniso}
\stepzero\noindstep\label{Ext1}
Let $X$ be an integral Noetherian normal $k$-scheme. Assume that $\bH\to X$ is a semisimple group scheme such that $\bH_{k(X)}$ is anisotropic. Let $\cE$ be a generically trivial $\bH$-torsor over $\A^1_k\times X$. Then $\cE$ can be extended to $\P^1_k\times X$ if and only if for all schematic points $x\in X$ the restriction of $\cE$ to $\A^1_k\times\Spec\cO_{X,x}$ can be extended to $\P^1_k\times\Spec\cO_{X,x}$.

\noindstep\label{Ext2}
The extension from part~\eqref{Ext1} is unique (if it exists) in the following sense. Let $\cE_1$ and $\cE_2$ be $\bH$-torsors over $\P^1_k\times X$ and let $\tau_i\colon\cE\to\cE_i|_{\A^1_k\times X}$ be isomorphisms. Then there is a unique isomorphism of $\bH$-torsors $\phi\colon\cE_1\to\cE_2$ such that $\phi|_{\A^1_k\times X}=\tau_2\circ\tau_1^{-1}$.
\end{proposition}
\begin{proof}
\eqref{Ext2} Since $\cE|_{\A^1_{k(X)}}$ is generically trivial, it is trivial by~\cite[Cor.~3.10(a)]{GilleTorseurs}; we choose a trivialization. Thus the extensions of $\cE|_{\A^1_{k(X)}}$ to $\P^1_{k(X)}$ are parameterized by sections of $\Gr_{\bH_{k(X)}}$ (see~\cite[Prop.~5.1]{FedorovExotic}). Since $\bH_{k(X)}$ is anisotropic, there is a unique such section (see~\cite[Sect.~6]{FedorovExotic}) and thus a unique extension.

Hence, the isomorphism $\tau_2\circ\tau_1^{-1}$ can be extended to the generic point of the infinity divisor $\infty\times X\subset\P^1\times X$. Since $\cE_i$ are affine over $X$ and $X$ is normal, by Hartogs' Theorem $\tau_2\circ\tau_1^{-1}$ can be extended to an isomorphism $\phi\colon\cE_1\to\cE_2$. Uniqueness of this extension follows because $\cE_i$ are affine and thus separated over $X$.

\eqref{Ext1} The condition is obviously necessary. Conversely, assume that it is satisfied and choose a trivialization of $\cE$ on $\A^1_k\times\Spec k(X)$ as above. Consider a point $x\in X$. Using our assumption, we can find a Zariski neighbourhood $U_x$ of $x$ and an extension of $\cE|_{\A^1_k\times U_x}$ to $\cE|_{\P^1_k\times U_x}$. By part~\eqref{Ext2}, for any points $x,y\in X$ the corresponding extensions are isomorphic on the intersection $\P^1_k\times(U_x\cap U_y)$ (in the sense made precise above). Note that these isomorphisms agree on triple intersections, by the uniqueness part of~\eqref{Ext2}. Thus the extensions glue together.
\end{proof}

We need a lemma.
\begin{lemma}\label{lm:NonExt}
  Let $R$ be a ring of positive Krull dimension. Then there is a morphism $\A_R^1\to\P^1_\Z$ that cannot be extended to $\P_R^1$.
\end{lemma}
\begin{proof}
  Write $\A_R^1=\Spec R[t]$. Since $R$ is of positive Krull dimension, there is a non-maximal prime ideal $\fp\subset R$. Then there is an element $x\in R$ such that its image in $R/\fp$ is neither zero, nor invertible. Consider the morphism $\A_R^1\to\P^1_\Z$ given by $(tx+1\colon1)$. We claim that it cannot be extended to $\P_R^1$. It is enough to show that its restriction to $\A_{R/\fp}^1$ cannot be extended to $\P_{R/\fp}^1$. Thus, replacing $R$ with $R/\fp$, we may assume that~$R$ is an integral domain and $x$ is a non-zero non-invertible element. Further, replacing $R$ with its normalization (note that non-invertible elements remain non-invertible in the normalization), we may assume that $R$ is normal. Replacing~$R$ with the localization at a maximal ideal containing $x$, we can further assume that $R$ is a normal local integral domain.

  Let us identify $\P^1_R-(0\times\Spec R)$ with $\Spec R[s]$, where $s=1/t$. In this chart, our morphism is defined on $\Spec R[s,s^{-1}]=\A^1_R-(0\times\Spec R)$ and is given by $(x+s\colon s)$. It is enough to show that it cannot be extended to $\A^1_R$. Since $R$ is a normal integral domain, by~\cite[Thm.~1]{BrewerCostaSeminormality} $\Pic(\A^1_R)=\Pic(R)$, which is trivial because $R$ is local. Therefore, such an extension would be of the form $(a\colon b)$, where $a,b\in R[s]$, $aR[s]+bR[s]=R[s]$, and $b(x+s)=as$. Since $bx=s(a-b)$ and $x$ is a non-zero element of $R$, we can write $b=sb'$ for some $b'\in R[s]$. Then $b'(x+s)=a$ and, since $x$ is not a unit in $R$,
  \[
    aR[s]+bR[s]\subset xR[s]+sR[s]\subsetneq R[s].
  \]
  This contradiction completes the proof.
\end{proof}

\begin{proof}[Proof of Proposition~\ref{pr:CounterExample}]
Let us denote an infinite field contained in $B$ by $k$. Since $\bG$ is semisimple, by~\cite[Exp.~XXIV, Thm.~4.1.5(i)]{SGA3-3} there is a finite \'etale cover $\Spec B'\to\Spec B$ and a split $k$-group scheme~$G$ such that $\bG_{B'}\simeq G\times_k\Spec B'$. Fix such $B'$ and an isomorphism $\bG_{B'}\simeq G\times_k\Spec B$.

Fix a closed $B$-embedding $\Spec B'\to\A^1_B$ (which exists because $B$ is local, $k$ is infinite, and $B'$ is $B$-finite and $B$-\'etale). Since $k$ is infinite, we may assume that this embedding factors through $\A^1_B-(0\times\Spec B)$.

For a scheme $S$, we denote the set of $S$-morphisms between $S$-ind-schemes $X_1$ and $X_2$ by $\Mor_S(X_1,X_2)$. We abbreviate $\Mor(X_1,X_2):=\Mor_{\Spec k}(X_1,X_2)$. We need one more lemma.

\begin{lemma}\label{lm:Grassmannians}
We have a bijection between $\Mor(\A^1_{B'},\Gr_G)$ and the set of pairs $(\cE,\tau)$, where $\cE$ is a $\bG$-torsor over $\A^1_k\times\P^1_B$, $\tau$ is a trivialization on $\A^1_k\times(\P^1_B-\Spec B')$.

Similarly, we have a bijection between $\Mor(\P^1_{B'},\Gr_G)$ and the set of pairs $(\cE,\tau)$, where $\cE$ is a $\bG$-torsor over $\P^1_k\times\P^1_B$, $\tau$ is a trivialization on $\P^1_k\times(\P^1_B-\Spec B')$. These two bijections are compatible in the obvious sense.
\end{lemma}
\begin{proof}
Recall that we have fixed a closed $B$-embedding $\Spec B'\to\A^1_B-(0\times\Spec B)$. Now~\cite[Prop.~5.1]{FedorovExotic} gives for any affine $B$-scheme $T$ a bijection
\begin{multline}\label{eq:Gr}
    \Mor_B(T\times_B\Spec B',\Gr_\bG)\simeq\\
    \{(\cE,\tau)\colon\text{$\cE$ is a $\bG$-torsor over $T\times_B\P_B^1$, and $\tau$ is its trivialization on $T\times_B(\P_B^1-\Spec B')$}\}.
\end{multline}
This is, in fact, an isomorphism of presheaves of sets on the category of affine $B$-schemes. Since both presheaves are actually sheaves in the Zariski topology, this isomorphism extends to an isomorphism of similar presheaves on the category of all $B$-schemes. We have
\[
        \Mor(\A^1_{B'},\Gr_G)=\Mor_{B'}(\A^1_{B'},\Gr_{G_{B'}})=\Mor_{B'}(\A^1_{B'},\Gr_{\bG_{B'}})=\Mor_B(\A^1_{B'},\Gr_\bG).
\]
Taking $T:=\A_B^1$ in~\eqref{eq:Gr} above, we see that the elements of $\Mor_B(\A^1_{B'},\Gr_\bG)$ correspond to pairs $(\cE,\tau)$, where $\cE$ is a $\bG$-torsor over $\A^1_B\times_B\P^1_B=\A^1_k\times\P^1_B$ and $\tau$ is its trivialization on
\[
    \A^1_B\times_B(\P^1_B-\Spec B')=\A^1_k\times(\P^1_B-\Spec B').
\]
The second bijection is constructed similarly upon taking $T=\P_B^1$ in~\eqref{eq:Gr}. The compatibility follows because our bijections come from an isomorphism of presheaves.
\end{proof}

We need one more lemma.
\begin{lemma}
Let, as above, $G$ be a split nontrivial semisimple group over a field $k$. Then the affine Grassmannian $\Gr_G$ contains a closed subscheme isomorphic to $\P_k^1$.
\end{lemma}
\begin{proof}
Since $G$ is split, we can choose a split maximal torus $T$ and a Borel subgroup $B$ such that $T\subset B\subset G$. Let $\lambda\colon\gm\to T$ be a nontrivial cocharacter. Consider the composition $\Spec k(\!(t)\!)\to\gm\xrightarrow{\lambda}T\to G$, where $k(\!(t)\!)$ is the field of Laurent series. This is a loop of $G$, denote by $t^\lambda$ its projection to $\Gr_G=G(\!(t)\!)/G[\![t]\!]$. By~\cite[Lm.~5.9]{FedorovExotic} the $G$-orbit of $t^\lambda$ in $\Gr_G$ is isomorphic to the flag variety of parabolic subgroups of type $\lambda$ of $G$, that is, to $G/P_\lambda$, where $P_\lambda\supset B$ is the standard parabolic subgroup of $G$ corresponding to $\lambda$.

It remains to show that $G/P_\lambda$ contains a closed subscheme isomorphic to $\P_k^1$. Let $\alpha$ be a simple root of $G$ with respect to $T\subset B$ such that the corresponding root subgroup is not contained in $P_\lambda$, and let $H$ be the subgroup of $G$ whose Lie algebra is the principal $\mathfrak{sl}_2$ corresponding to $\alpha$. Then $H$ is isomorphic to $\SL_2$ or $\PGL_2$ and $H\cap P_\lambda$ is a nontrivial parabolic subgroup of $H$. Thus, we have a closed embedding $\P^1_k=H/(H\cap P_\lambda)\to G/P_\lambda$.
\end{proof}

We can now complete the proof of Proposition~\ref{pr:CounterExample}. Clearly, $B'$ is not Artinian. By Lemma~\ref{lm:NonExt} there is a morphism $\A^1_{B'}\to\P^1_\Z$ that cannot be extended to
$\P^1_{B'}$. The corresponding $k$-morphism $\A^1_{B'}\to\P^1_k$ cannot be extended to $\P^1_{B'}$ either. Let $\phi\colon\A^1_{B'}\to\Gr_G$ be the composition of this morphism with the constructed above closed embedding $\P_k^1\hookrightarrow\Gr_G$. By construction, $\phi$ cannot be extended to $\P^1_{B'}$.

By Lemma~\ref{lm:Grassmannians}, $\phi$ corresponds to a $\bG$-torsor $\cE$ over $\A^1_k\times\P^1_B$ with a trivialization $\tau$ on $\A^1_k\times(\P^1_B-\Spec B')$. For a point $x\in\P^1_B$ let $R_x$ be the local ring of $x$ on $\P^1_B$. We claim that for some $x$ the $\bG$-torsor $\cE|_{\A^1_{R_x}}$ cannot be extended to $\P^1_{R_x}$. Assume the converse. Set $\bH=\bG_{\P^1_B}$, then $\bH$ is generically anisotropic. Indeed, were it isotropic over some open subscheme $U\subset\P^1_B$, there would be a rational point $y\in\P_k^1$ such that $(y\times\Spec B)\cap U\ne\emptyset$ (because $k$ is infinite). But then $\bG\approx\bH|_{y\times\Spec B}$ would not be generically anisotropic. Applying Proposition~\ref{pr:ExtAniso}(i) to $X=\P^1_B$ and $\bH$, we get an extension of $\cE$ to a $\bG$-torsor $\widetilde\cE$ over $\P^1_k\times\P^1_B$.

The trivial $\bG$-torsor $\cE|_{\A^1_k\times(\P^1_B-\Spec B')}$ can obviously be extended to a trivial torsor on $\P^1_k\times(\P^1_B-\Spec B')$.
Applying Proposition~\ref{pr:ExtAniso}(ii) to $X=\P^1_B-\Spec B'$, $\bH=\bG_{\P^1_B-\Spec B'}$ (which is also generically anisotropic), we see that this trivialization gives an extension $\tilde\tau$ of the trivialization $\tau$. The pair $(\widetilde\cE,\tilde\tau)$ gives an extension of $\phi$ to $\P^1_{B'}$ by Lemma~\ref{lm:Grassmannians}. However, $\phi$ cannot be extended to $\P^1_{B'}$, and we get a contradiction. Without loss of generality the point $x$ belongs to $\A^1_B\subset\P^1_B$, so the proposition is proved.
\end{proof}

\bibliographystyle{../../alphanum}
\bibliography{../../RF}

\end{document}